\documentclass[11pt]{article}

\usepackage{amssymb, amsmath, amsthm, graphicx}
\usepackage[left=1in,top=1in,right=1in]{geometry}

\usepackage{subfigure}

      \theoremstyle{plain}
      \newtheorem{theorem}{Theorem}[section]
      \newtheorem{lemma}[theorem]{Lemma}
            \newtheorem{claim}[theorem]{Claim}
      \newtheorem{corollary}[theorem]{Corollary}
            \newtheorem{observation}[theorem]{Observation}

      \theoremstyle{definition}

      \theoremstyle{remark}

\title{Erd\H os-Szekeres-type theorems\\ for monotone paths and convex bodies}

\author{Jacob Fox \thanks{Massachusetts Institute of Technology, Cambridge, MIT. Supported by a Simons Fellowship. Email:{\tt fox@math.mit.edu}} \and J\'anos Pach\thanks{City College and Courant Institute, New
York and EPFL, Lausanne. Supported by Grants from NSF, NSA, PSC-CUNY, BSF, OTKA, and SNF.
Email: {\tt pach@cims.nyu.edu}} \and Benny Sudakov\thanks{Department of Mathematics,
UCLA,  Los Angeles, CA 90095. Email: {\tt bsudakov@math.ucla.edu}. Research
supported in part by NSF CAREER award DMS-0812005 and by a
USA-Israeli BSF grant.} \and Andrew Suk\thanks{Courant Institute, New York and EPFL, Lausanne. Email: {\tt suk@cims.nyu.edu}}
}

\begin{document}

\maketitle

\begin{center}{\small
\em{Dedicated to the 75th anniversary of the publication of the Happy Ending Theorem}}
\end{center}

\medskip

\begin{abstract}

For any sequence of positive integers $j_1 < j_2 < \cdots < j_n$, the $k$-tuples $(j_i,j_{i + 1},...,j_{i + k-1}),$
$i=1, 2, \ldots, n - k+1$, are said to form a {\em monotone path} of length $n$. Given any integers $n\ge k\ge 2$ and $q\ge 2$, what is the smallest integer $N$ with the property that no matter how we color all $k$-element subsets of $[N]=\{1,2,\ldots, N\}$ with $q$ colors, we can always find a monochromatic monotone path of length $n$? Denoting this minimum by $N_k(q,n)$, it follows from the seminal 1935 paper of Erd\H os and Szekeres that $N_2(q,n)=(n-1)^q+1$ and $N_3(2,n) = {2n -4\choose n-2} + 1.$ Determining the other values of these functions appears to be a difficult task. Here we show that
$$2^{(n/q)^{q-1}} \leq N_3(q,n) \leq 2^{n^{q-1}\log n},$$ for $q \geq 2$ and $n \geq q+2$. Using a ``stepping-up" approach that goes back to Erd\H os and Hajnal, we prove analogous bounds on $N_k(q,n)$ for larger values of $k$, which are towers of height $k-1$ in $n^{q-1}$. As a geometric application, we prove the following extension of the Happy Ending Theorem. Every family of at least $M(n)=2^{n^2 \log n}$ plane convex bodies in general position, any pair of which share at most two boundary points, has $n$ members in convex position, that is, it has $n$ members such that each of them contributes a point to the boundary of the convex hull of their union.
\end{abstract}

\section{Introduction}
The classic 1935 paper of Erd\H os and Szekeres \cite{e} published in {\em Compositio Mathematica} was a starting point of a very rich discipline within combinatorics: Ramsey theory (see, e.g., \cite{graham}). Erd\H os liked to call the main result of this paper the ``Happy Ending Theorem," as its discovery was triggered by a geometric observation of Esther Klein, and the authors' collaboration with her eventually led to the marriage of Klein and Szekeres.

At the early stages of its development, Ramsey theory focused on the emergence of large monochromatic {\em complete} subgraphs in colored graphs and hypergraphs. In the 1960s and 70s, many initial results of Ramsey theory were revisited in a more general setting. For instance, Gerencs\'er and Gy\'arf\'as \cite{GeGy67} proved that for any 2-coloring of the edges of a complete graph with roughly $3n/2$ vertices, there is a monochromatic path of length $n$. Recently, Figaj and {\L}uczak \cite{FiLu07} and, in a more precise form, Gy\'arf\'as et al.~\cite{GyR07} have settled an old conjecture of Faudree and Schelp, according to which any 3-colored complete graph with roughly $2n$ vertices has a monochromatic path of length $n$. It follows from a theorem of Erd\H os and Gallai~\cite{EG59} that for every positive integer $q$ there exists a constant $c(q)\le q$ such that no matter how we color the edges of the complete graph $K_{c(q)n}$ with $q$ colors, we can always find a monochromatic path of length $n$. The smallest value of the constant $c(q)$ is not known.

The problem becomes simpler if we consider {\em ordered} complete graphs and we want to find a long {\em increasing} path such that all of its edges are of the same color. It follows from the results of Erd\H os and Szekeres \cite{e}, and also from Dilworth's theorem on partially ordered sets \cite{dil}, that for any 2-coloring of the edges of a complete graph with more than $(n-1)^2$ vertices, one can find a monochromatic monotone path with $n$ vertices. An easy construction shows that this statement does not remain true for graphs with $(n-1)^2$ vertices. The result and its proof readily generalize to colorings with $q\geq 2$ colors: in this case, we have to replace $(n-1)^2$ with $(n-1)^q$.

\medskip

The corresponding problems for hypergraphs are quite hard and their solution usually utilizes some variant or extension of the celebrated hypergraph regularity lemma. The only asymptotically sharp result of this kind was proved by Haxell, {\L}uczak et al.~\cite{HaLu09}. Let $H=(V(H),E(H))$ be a $k$-uniform hypergraph. A \emph{tight path} or, shortly, a {\em path of length} $n$ in $H$ is comprised of a set of $n$ distinct vertices $v_1, v_2,\ldots,v_n\in V(H)$ and the set of $n-k+1$ hyperedges of the form
$$\{v_i,v_{i + 1},...,v_{i + k - 1}\} \in E(H),$$
where $i=1, 2, \ldots, n - k+1$. It was shown in~\cite{HaLu09} that no matter how we 2-color all triples of a complete 3-uniform hypergraph on $4n/3$ vertices, one of the color classes will contain a path of length roughly $n$. This result is best possible. It was generalized by Nagle, Olsen et al.~\cite{NaOl08},  Cooley, Fountoulakis et al.~\cite{CoFo09}, and Conlon et al.~\cite{conlon}. These results imply that every $q$-coloring of the hyperedges of a complete $k$-uniform hypergraph on $Cn$ vertices contains a monochromatic copy of any $n$-vertex hypergraph with maximum degree $\Delta$, where $C=C(k,q,\Delta)$ is a suitable constant depending only on $k$, $q$, and $\Delta$.

\medskip

In the present paper, we consider the analogous problem for {\em ordered} hypergraphs. Let $K^k_N$ denote the {\em complete $k$-uniform hypergraph}, consisting of all $k$-element subsets ($k$-tuples) of the vertex set $[N] = \{1,2,...,N\}$. For any $n$ positive integers, $j_1 < j_2 < \cdots < j_n$, we say that the hyperedges
$$\{j_i,j_{i + 1},...,j_{i + k-1}\},$$
$i=1, 2, \ldots, n - k+1$, form a {\em monotone (tight) path of length} $n$.  We will present several results on finding monochromatic monotone paths in edge-colored ordered hypergraphs.  As a geometric application, in Section 7 we will show that every family of at least $2^{n^2 \log n}$ plane convex bodies in general position, any pair of which share at most two boundary points, has $n$ members in convex position. This substantially improves the previous double-exponential bound of Hubard et al. \cite{alfredo} for this problem. Here, and throughout the paper, all logarithms unless otherwise stated are in base 2.

\subsection{Ramsey numbers for monotone paths}

Let $N_k(q,n)$ denote the smallest integer $N$ such that for every $q$-coloring of the hyperedges ($k$-tuples) of $K^k_N$, there exists a monotone path of length $n$ such that all of its hyperedges
are of the same color. Using this notation for graphs $(k=2)$, the corollary of the Erd\H os-Szekeres theorem or the Dilworth theorem mentioned above can be rephrased as
\begin{equation}\label{dil}
N_2(q,n)=(n-1)^q+1.
\end{equation}

For 3-uniform hypergraphs, a straightforward generalization of the cup-cap argument of Erd\H os and Szekeres \cite{e} gives

\begin{theorem}\label{erdos-szekeres}
Let $N=N_3(2,n)$ be the smallest integer such that for every $2$-coloring of all triples in $[N]$, there exists a monochromatic monotone path of length $n$. Then we have
 $$N_3(2,n) = {2n -4\choose n-2} + 1.$$
\end{theorem}

For a larger (but fixed) number of colors, we prove the following theorem which provides lower and upper bounds that are tight apart from a logarithmic factor in the exponent.

\begin{theorem}
\label{main1}
For any integers $q \ge 2$ and $n \ge q+2$, we have
$$2^{(n/q)^{q-1}} \leq N_3(q,n) \leq 2^{n^{q-1}\log n}.$$
\end{theorem}

Define the tower function $t_i(x,n)$ recursively as follows. Let $t_1(x,n)=x$ and $t_{i+1}(x,n)=n^{t_i(x,n)}$, so $t_i(x,n)$ is a tower of $i-1$ $n$'s with an $x$ on top. We let $t_i(x)=t_i(x,2)$.  Using a ``stepping-up" approach, developed by Erd\H os and Hajnal (see~\cite{graham}) and strengthened in~\cite{conlon}, we show the following extension of Theorem \ref{main1} to $k$-uniform hypergraphs, for any $k\geq 3$.

\begin{theorem}\label{main2}
For every $k \geq 3$ and $q$, there are positive constants $c,c'$ depending only on $k$ and $q$ such that, for any $n \geq 4k$, we have
$$t_{k-1}(cn^{q-1}) \leq N_k(q,n) \leq t_{k-1}(c'n^{q-1}\log n).$$
\end{theorem}

In Section~\ref{Section2}, we prove the recursive upper bound $N_k(q,n) \leq N_{k-1}((n-k+1)^{q-1},n)$. This upper bound, together with equation (\ref{dil}), implies that $N_k(q,n) \leq t_k(q-1,n)$ for $k \geq 3$. The upper bounds in Theorems \ref{main1} and \ref{main2} follow from this inequality.

In Sections~\ref{section3} and~\ref{Section4}, we establish the lower bound in Theorem \ref{main1}. We also prove a general statement there (Theorem \ref{stepup4}), providing a lower bound on $N_k(q,n)$, which is exponential in $N_{k-1}(q,n)$, for every $k\ge 4$. Putting these results together, the lower bound in Theorem \ref{main2} readily follows.

The case $k=3$ is crucial to understanding the growth of $N_k(q,n)$. Indeed, the stepping-up lower bound mentioned above (see Theorem \ref{stepup4}) together with the recursive upper bound in Theorem \ref{alternativeupperbound} and inequality \ref{trivial123}, described in the next subsection, show that closing the gap between the upper and lower bounds on $N_k(q,n)$ in the special case $k=3$ would also close the gap for all larger values of $k$.

\subsection{Online and size Ramsey numbers}\label{online}

Consider the following game played by two players, {\em builder} and {\em painter}. For $t \geq 1$, at the beginning of stage $t$, a new vertex $v_{t}$ is added (the vertices $v_i$ for $1 \leq i < t$ already are present), and for each $(k-1)$-tuple of vertices $(v_{i_1},\ldots,v_{i_{k-1}})$ with $1 \leq i_1<\ldots<i_{k-1}<t$, builder decides (in any order) whether to draw the edge $(v_{i_1},\ldots,v_{i_{k-1}},v_{t})$. If builder draws the edge, then painter has to immediately
color it in one of $q$ colors $1,\ldots,q$. The {\it (vertex) online Ramsey number} $V_k(q,n)$ is the minimum number of edges builder has to draw to guarantee a monochromatic monotone path of length $n$.

An ordered $k$-uniform hypergraph $G$ is said to be {\it $(q,n)$-path Ramsey} if for every $q$-edge-coloring of $G$ one can find a monochromatic monotone path of length $n$.
The {\it size Ramsey number} $S_k(q,n)$ is the minimum number of edges of an ordered $k$-uniform hypergraph $G$ which is $(q,n)$-path Ramsey. It follows from the definitions that

\begin{equation}\label{trivial123}
N_k(q,n-k+1) \leq V_k(q,n) \leq S_k(q,n) \leq {N_k(q,n) \choose k}.\end{equation}
Indeed, the first inequality comes from the following painter strategy. Let $w_i$ be the $i$th vertex that builder adds in which there is an edge whose largest vertex is $w_i$ (there may be vertices not given labels). That is, $w_i=v_{j(i)}$ where $j(i)$ is the $i$th stage for which at least one edge is added. Painter colors each edge $(w_{i_1},w_{i_2},\ldots,w_{i_k})$ with $i_1<\ldots<i_k < N_k(q,n-k+1)$ the same color as the color of $(i_1,\ldots,i_k)$ in a $q$-edge-coloring of the ordered complete $k$-uniform hypergraph on $N_k(q,n-k+1)-1$ vertices with no monochromatic path of length $n-k+1$. All other edges, i.e., those containing a vertex which is not the largest vertex in some edge, can be colored arbitrarily. Note that every  monotone path of length $n$ contains a path of length $n-k+1$ such that each vertex is the largest vertex in an edge of the ordered hypergraph. Therefore, this coloring has no monochromatic monotone path of length $n$, and painter guarantees that there are at least $N_k(q,n-k+1)$ edges when the first monochromatic monotone path of length $n$ appears.  On the other hand, if builder selects all edges of a $(q,n)$-path Ramsey ordered $k$-uniform hypergraph, she will definitely win; which proves the second inequality. The third inequality is trivial, taking into account that the ordered complete $k$-uniform hypergraph $K^k_N$ with $N=N_k(q,n)$ is $(q,n)$-path Ramsey.

Conlon, Fox, and Sudakov \cite{CFS10} used online Ramsey numbers to give an upper bound for the Ramsey number of a {\em complete} $k$-uniform hypergraph with $n$ vertices. Here we apply this technique to establish an upper bound for the Ramsey number of a {\em monotone path} of length $n$. More precisely, we show

\begin{theorem}\label{alternativeupperbound}
For every $k\ge 3$ and for every $q\ge 1,\; n\ge k$, we have
$$N_k(q,n) \leq q^{V_{k-1}(q,n+k-2)} + k-2.$$
\end{theorem}

In the special case $k=3$, this bound is worse than the upper bound in Theorem \ref{main1}. However, Theorem \ref{alternativeupperbound} is already useful for $k \geq 4$. Indeed, any improvement on the upper bound in Theorem \ref{main2} for $k=3$, together with the trivial upper bound on $V_{k-1}(q,n+k-2)$ in (\ref{trivial123}), will lead to analogous improvements for all larger $k$. For example, if, for fixed $q$, we would know that $N_3(q,n)=2^{O(n^{q-1})}$, then using Theorem \ref{alternativeupperbound} we would also know that $N_4(q,n) =2^{2^{O(n^{q-1})}}$. Note that the recursive bound in Theorem \ref{upineq} does not give such an improvement even if we knew that $N_3(q,n)$ is exponential in $n^{q-1}$. Similarly, the stepping-up lower bound in  Theorem \ref{stepup4} shows that any improvement on the lower bound in Theorem \ref{main2} for $k=3$ will lead to analogous improvements for all larger $k$.

In addition to being useful for bounding classical Ramsey numbers as discussed above, online and size Ramsey numbers are now well-studied topics in their own right. The {\it size Ramsey number} $r_e(H)$ of a graph $H$ is the minimum number of edges of a graph $G$ in which every $2$-edge-coloring of $G$ contains a monochromatic copy of $H$. The study of these numbers was initiated by Erd\H{o}s et al.~\cite{EFRS}. A fundamental problem of Erd\H{o}s in this area is to determine the growth of the size Ramsey number of paths.  Beck \cite{Be} solved this problem, proving that these numbers grow linearly. That is, $r_e(P_n) \leq c n$, where $c$ is an absolute constant. In contrast, R\"odl and Szemer\'edi \cite{RS} disproved a conjecture of Beck by showing that there are graphs of maximum degree $3$ whose size Ramsey number grows superlinear in the number of vertices. In the other direction, Kohayakawa et al.~\cite{KRSS} recently showed that graphs on $n$ vertices with fixed maximum degree $\Delta$ have size Ramsey number at most $O(n^{2-1/\Delta}\log^{1/\Delta} n)$.

Another online Ramsey game which is quite close to ours was introduced independently by Beck \cite{Be2} and Kurek and Ruci\'nski \cite{KR}. In this game, there are two players, Builder and Painter, who move on the originally empty graph with an unbounded number of vertices. At each step, Builder draws a new edge and Painter has to color it either red or blue immediately. The {\it edge online Ramsey number} $\bar r(H)$ is the minimum number of edges that builder has to draw in order to force painter to create a monochromatic $H$. As Builder can simply draw the edges of a given graph, we have the inequality $\bar r(H) \leq r_e(H)$. A basic conjecture in this area due to R\"odl is to show that $\lim_{n \to \infty} \frac{\bar r(K_n)}{r_e(K_n)}=0$. Conlon \cite{Co} made substantial progress on this conjecture, showing that $\bar r(K_n) \leq c^{n}r_e(K_n)$ holds infinitely often, where $c<1$ is an absolute constant. Randomized variants of the edge online Ramsey number were studied in \cite{BB}, \cite{BMS}, \cite{FKRRT}.

Note that the problem of estimating $V_k(q,n)$ and $S_k(q,n)$ is most interesting in the case of graphs ($k=2$). Indeed, for larger $k$, $N_k(q,n)$ grows roughly $(k-2)$-fold exponential in $n^{q-1}$, so the lower and upper bounds on these numbers are roughly determined by $N_k(q,n)$. We will therefore focus our attention on the interesting case $k=2$.

In view of (\ref{trivial123}), one may think that the functions $V_k(q,n)$ and $S_k(q,n)$ cannot differ too much. Rather surprisingly, already for graphs ($k=2$) this is not the case. In Section~\ref{sectionsizeandvertex}, we show that the size Ramsey number satisfies $S_2(q,n) \geq c_qn^{2q-1}$. On the other hand, also in that section we prove

\begin{theorem}\label{keyonlinethm}
We have $$V_2(2,n)=(1+o(1))n^2\log_2 n,$$
and for every fixed $q \geq 2$, there are constants $c_q$ and $c_{q}'$ such that
$$c_q n^q \log n \leq V_2(q,n)\leq c_q' n^q \log n.$$
\end{theorem}

Moreover, the proof of Theorem \ref{keyonlinethm} shows that builder has a winning strategy which uses $N_2(q,n)$ vertices and each vertex $v_t$ belongs to at most $c_q'\log n$ edges $(v_j,v_t)$ with $j<t$. This theorem shows that, by using the information on the initial portion of the coloring, builder can substantially reduce the number of edges that guarantee a monochromatic monotone path.

\subsection{Noncrossing convex bodies in convex position}\label{noncrossing}

In the last section, we apply Theorem \ref{main1} and Lemma \ref{clique} to a problem for families of convex bodies. To formulate our question, we need some definitions. A family $\mathcal{C}$ of convex bodies (compact convex sets with nonempty interior) in the plane is said to be in \emph{convex position} if none of its members is contained in the convex hull of the union of the others. We say that $\mathcal{C}$ is in \emph{general position} if
\begin{enumerate}
\item every three members of $\mathcal{C}$ are in convex position;
\item no two members $C, C'\in\mathcal{C}$ have a common tangent that meets $C\cap C'$; and
\item no three members of $\mathcal{C}$ share a common tangent.
\end{enumerate}

Bisztriczky and Fejes-T\'oth~\cite{bf1} generalized the Happy Ending theorem of Erd\H os and Szekeres~\cite{e}, \cite{bf2} as follows. They proved that for every $n$, there exists a function $D(n)$ such that any family of $D(n)$ pairwise disjoint convex bodies in general position in the plane contains $n$ members in convex position. It was proved in \cite{pt} that $D(n) \leq \left({2n-4\choose n-2} + 1\right)^2$.

\smallskip

In \cite{pt2}, the Bisztriczky-Fejes T\'oth theorem has been extended to families of \emph{noncrossing} convex bodies, that is, to convex bodies, no pair of which share more than {\em two} boundary points. It was proved that there exists a function $N=M(n)$ such that from every family of $N$ noncrossing convex bodies in general position in the plane one can select $n$ members in convex position. More recently, Hubard et al.~\cite{alfredo} showed that the function $M(n)$ grows at most double exponentially in $n$.  We use Theorem \ref{main1} to obtain a much better bound.

\begin{theorem}
\label{convex}
Let $N = M(n)$ denote the smallest integer such that every family of $N$ noncrossing convex bodies in general position in the plane has $n$ members in convex position. Then we have $$M(n) \leq N_3(3,n) \leq 2^{n^2\log n}.$$
\end{theorem}

Note that no analogue of the last theorem is true if we drop the assumption that the bodies are noncrossing. One can construct a family of arbitrarily many pairwise crossing rectangles that which
are in general position, but no four of them are in convex position (see \cite{pt2}).

\medskip

Theorem \ref{convex} can be established, as follows. First, we choose a coordinate system, in which the {\em left endpoint} (i.e., the leftmost point) of every member of the family is unique, and the $x$-coordinates of the left endpoints are different. Let $C_1, C_2,\ldots, C_N$ denote the members of our family, listed in the increasing order of the $x$-coordinates of their left endpoints. Suppose that $N=N_3(3,n)$.

In the last section, we define a 3-coloring of all ordered triples $(C_i,C_j,C_k)$ with $i<j<k$. We will show that this coloring has the property that if a sequence of sets $C_{i_1}, C_{i_2},\ldots \;$ $(i_1<i_2<\cdots)$ induces a monochromatic monotone path, then they are in convex position.

According to Theorem \ref{main1}, $N \leq n^{n^2}.$  Therefore, our colored triple system contains a monochromatic monotone path of length $n$. In other words, there exists a sequence $i_1<i_2<\cdots<i_n$ such that all ordered triples induced by the sets $C_{i_j}$ are of the same color. It follows from the special property of our coloring mentioned above that the sets $C_{i_1},\ldots, C_{i_n}$ are in convex position. This will complete the proof of Theorem \ref{convex}.

\vspace{0.1cm}
\noindent {\bf Organization:} The rest of the paper is organized as follows. In the next section we prove a recursive upper bound on $N_k(q,n)$. In Section \ref{section3}, we prove results on monochromatic walks in digraphs and the minimum chromatic number of a $(q,n)$-path Ramsey graph which will be used in the subsequent two sections. In Section \ref{Section4}, we prove lower bounds on $N_k(q,n)$. In Section \ref{sectionsizeandvertex}, we prove bounds on size and online Ramsey numbers of paths, as well as an alternative upper bound on $N_k(q,n)$ in terms of online Ramsey numbers. In Section \ref{transitivesect}, we prove a lemma relating monochromatic cliques in ``transitive'' colorings to monochromatic monotone paths in general colorings. We use this lemma in Section \ref{geomsect} in the proof of Theorem \ref{convex}. We finish with some concluding remarks and open problems in Section \ref{concluding}.

\section{An upper bound on $N_k(q,n)$}\label{Section2}

Given any two positive integers $i \leq j$, let $[i,j]$ denote the set of positive integer $h$ with $i \leq h \leq j$, and $[i]=[1,i]$. We begin by proving a recursive upper bound on $N_k(q,n)$.

\begin{theorem}\label{upineq}
For any integers $n\ge k\ge 2$ and $q\ge 2$, we have $$N_k(q,n) \leq N_{k-1}((n-k+1)^{q-1},n).$$
\end{theorem}
\medskip
\noindent{\bf Proof.} Suppose for contradiction that there is a $q$-coloring $$\chi:{[N] \choose k} \rightarrow [q]$$ of the edges of the complete $k$-uniform ordered hypergraph  $K^k_N$ on $N=N_{k-1}((n-k+1)^{q-1},n)$ vertices with no monochromatic path of length $n$. Define the auxiliary coloring $$\phi: {[N] \choose k-1} \rightarrow [k-1,n-1]^{q-1},$$ as follows. For any $v_1<\ldots<v_{k-1}$, we let $\phi(v_1,\ldots,v_{k-1})=(n_1,\ldots,n_{q-1})$, where $n_i$ is the length of the longest monochromatic path in color $i$ ending with $v_1,\ldots,v_{k-1}$. By convention, any $k-1$ vertices $v_1,\ldots,v_{k-1}$ form a monochromatic monotone path of length $k-1$ with respect to coloring $\chi$. Thus, we have $k-1\leq n_i \leq n-1$, for every $i\; (1\le i\le q-1)$.

By our assumption, $N = N_{k-1}((n-k+1)^{q-1},n)$, so that in coloring $\phi$, there is a monochromatic monotone path of length $n$. Let $u_1<\ldots < u_n$ denote the vertices of this path. To complete the proof of the theorem, it is sufficient to show that
these vertices form a monochromatic monotone path in color $q$, with respect to the coloring $\chi$. Suppose that this is not the case. Then, for some $1 \leq j \leq n-k+1$ and $1 \leq i \leq q-1$, we have $\chi(u_{j},u_{j+1},\ldots,u_{j+k-1})=i$. Since the vertices $u_1,\ldots,u_n$ form a monochromatic monotone path with respect to the coloring $\phi$, it follows that in coloring $\chi$, the length of the longest monochromatic monotone path of color $i$ ending with $u_j,u_{j+1},\ldots,u_{j+k-2}$ is the same as the length of the longest monochromatic monotone path of color $i$ ending with $u_{j+1},u_{j+2},\ldots,u_{j+k-1}$. However, any longest monochromatic monotone path of color $i$ ending with $u_j,u_{j+1},\ldots,u_{j+k-2}$ can be extended to a longer monochromatic monotone path of color $i$, by adding the vertex $u_{j+k-1}$. This contradiction completes the proof of the theorem.
\qed

\medskip

Now we are in a position to prove the upper bound in Theorem \ref{main1}.

\medskip

\noindent {\bf Proof of the upper bounds in Theorems \ref{main1} and \ref{main2}.} We prove the stronger inequality  $$N_k(q,n) \leq t_k(q-1,n),$$ for every $k \geq 3$.
The proof is by induction on $k$. By (\ref{dil}), we have $N_2(q,n)=(n-1)^{q}+1$. According to Theorem \ref{upineq}, $$N_3(q,n) \leq N_2((n-2)^{q-1},n) = (n-1)^{(n-2)^{q-1}}+1<n^{n^{q-1}}=t_3(q-1,n).$$
Now suppose that the desired inequality holds for $k$. Then it also holds for $k+1$, because, again by Theorem \ref{upineq}, we have
$$N_{k+1}(q,n) \leq N_k((n-k)^{q-1},n) \leq t_k((n-k)^{q-1},n) \leq t_k(n^{q-1},n)=t_{k+1}(q-1,n).$$
\qed

\section{Monochromatic walks in digraphs}\label{section3}
Let $D_N$ denote the {\em complete digraph} on $N$ vertices, that is, the directed graph in which each pair of distinct vertices is connected by two edges with opposite orientations. A {\em walk of length $n$} in a digraph is a sequence of $n$ vertices $v_1,\ldots, v_n$ with possible repetitions such that for every $i\; (1\le i<n)$, the directed edge $\overrightarrow{v_iv_{i+1}}$ belongs to the digraph. If $v_1=v_n$, then the walk is called {\em closed}. Note that closed walks can be used to construct walks which are arbitrarily long. A digraph with no closed walk is {\em acyclic}.

Let $f(q,n)$ be the smallest number $N$ such that for every $q$-coloring of the edges of $D_N$, there is a monochromatic walk of length $n$. First, we show that for a fixed $q$, the order of magnitude of $f(q,n)$ is $n^{q-1}$. In the next section, this fact is used to establish the lower bound on $N_3(q,n)$, stated in Theorem~\ref{main1}.

\begin{theorem}\label{ftheorem} For any integers $n, q\ge 2$,
we have $$(n/q)^{q-1} \leq f(q,n) \leq N_2(q-1,n)=(n-1)^{q-1}+1.$$
\end{theorem}

\medskip
\noindent{\bf Proof.}
We first establish the upper bound. Consider a $q$-coloring of the edges of $D_N$ with $N=N_2(q-1,n)$. We have to show that there is a monochromatic walk with $n$ vertices. Suppose that the set of edges of color $q$ does not determine a walk of length $n$. Then these edges form an acyclic digraph. Hence, there is an ordering of the vertices of $D_N$ so that the edges of color $q$ go backwards. Thus, removing the backwards edges, we have an ordered complete graph on $N = N_2(q-1,n)$ vertices, whose edges are colored with $q-1$ colors. In one of the color classes, we can find a monotone path with $n$ vertices, which is a monochromatic directed path in $D_N$.

The lower bound follows from Lemma \ref{lowf} below. Indeed, letting $m=1+\lfloor \frac{n-2}{q-1} \rfloor$, it is easy to check that
$2+(q-1)(m-1) \leq n$ and $m \geq n/q$, so $$f(q,n) \geq f(q,2+(q-1)(m-1)) \geq m^{q-1} \geq (n/q)^{q-1}.$$

\qed

\smallskip

\begin{lemma}\label{lowf}
For any integers $n,q\ge 2$, we have $f(q,2+(q-1)(n-1)) > n^{q-1}$.
\end{lemma}

\medskip
\noindent{\bf Proof.}
Define the coloring $\phi$ of the edges of the complete digraph on $[n]^{q-1}$, as follows. For any pair of distinct vertices, $a=(a_1,\ldots,a_{q-1}),b=(b_1,\ldots,b_{q-1})$, define the color of the directed edge $(a,b)$ as the smallest coordinate $i$ for which $a_i<b_i$. If there is no such coordinate, that is, if $a_i \geq b_i$ for all $i\; (1 \leq i \leq q-1)$), then color the directed edge $(a,b)$ with color $q$. For each color $i \leq q-1$, the length of a longest monochromatic walk of color $i$ is $n$, because traversing any edge of color $i$, the $i$th coordinate must increase. On the other hand, along any walk in color class $q$, the coordinates are never allowed to increase, and in each step at least one of them must strictly decrease, so that the sum of the coordinates strictly decreases.  The sum of the coordinates of a point in $[n]^{q-1}$ is at least $q-1$ and is at most $(q-1)n$. Therefore, the length of such a monochromatic walk does not exceed $(q-1)n-(q-1)+1=1+(q-1)(n-1)$.
\qed

\medskip
Recall from Subsection~\ref{online} that an ordered graph $G$ is called {\em $(q,n)$-path Ramsey} if for every $q$-edge-coloring of $G$ there exists a monochromatic monotone path of length $n$. Let $\chi(q,n)$ denote the minimum chromatic number of an ordered graph $G$ which is $(q,n)$-path Ramsey.

We close this section by showing that the functions $\chi(q,n)$ and $f(q,n)$ are actually identical. Therefore, Theorem~\ref{ftheorem} implies that, for a fixed $q$, $\chi(q,n)$ also grows on the order of $n^{q-1}$.

The (classical) {\em Ramsey number} $R(n;q)$ is the minimum $N$ such that every $q$-edge-coloring of the complete graph on $N$ vertices contains a monochromatic clique on $n$ vertices.

\smallskip

\begin{theorem}\label{chif}
For any integers $q,n\ge 2$, we have $$\chi(q,n) = f(q,n).$$
\end{theorem}

\medskip
\noindent{\bf Proof.}
We first show that $\chi(q,n) \geq f(q,n)$. Consider a $q$-edge-coloring $\phi$ of the edges of the complete digraph with vertex set $[f(q,n)-1]$ without a monochromatic walk of length $n$. Assume for contradiction that there exists an ordered graph $G$ with $\chi(G) < f(q,n)$, which is $(q,n)$-path Ramsey, and consider a partition $V(G)=V_1 \cup \ldots \cup V_t$ into independent sets with $t=\chi(G)$. Define a $q$-coloring $\rho$ of the edges of $G$, as follows. If $v,w$ are adjacent with $v<w$ and $v \in V_i$ and $w \in V_j$, let  $\rho(v,w)=\phi(i,j)$. If the vertices $v_1<\ldots<v_n$ form a monochromatic monotone path with respect to the edge-coloring $\rho$ of $G$, then denoting by $i_k$ the integer for which $v_k \in V_{i_k}$, we have that $i_1,\ldots,i_n$ form a monochromatic walk with respect to the edge-coloring $\phi$. This is a contradiction, showing that our assumption $\chi(G) < f(q,n)$ was wrong.

To show that $\chi(q,n) \leq f(q,n)$, we define an ordered graph $H$, which is $f(q,n)$-colorable and $(q,n)$-path Ramsey. Let $V(H)=[Rt]$, where $t=f(q,n)$ and $R=R(n;Q)$ is the $Q$-color Ramsey number for the complete graph on $n$ vertices with $Q=q^{t^2-t}$. Connect two vertices $i,j\; (1\le i<j\le Rt)$ by an edge in $H$ if $j-i$ is not a multiple of $t$. The partition of $V(H)$ modulo $t$ defines a proper $t$-coloring of $H$. Since the first $t$ vertices of $H$ form a clique, we have $\chi(H) =t$. It remains to prove that $H$ is $(q,n)$-path Ramsey. Consider an edge-coloring $\psi$ of $H$ with $q$ colors. Define an auxiliary $Q$-coloring $\tau$ of the complete graph on $[R]$ vertices, where the color of the edge between $u, v\; (1\le u<v\le R)$ is given as a $t \times t$ matrix $A=(a_{ij})$ with $0$'s in the diagonal and $a_{ij}=\psi((u-1)t+i,(v-1)t+j)$ for every $i\not=j$. Obviously, the number of colors used in this coloring is at most the number of possible matrices $A$, which is equal to $q^{t^2-t}=Q$. By the definition of the Ramsey number $R=R(n;Q)$, $[R]$ contains a monochromatic clique on $n$ vertices, with respect to the coloring $\tau$. Denote the vertices of such a clique by $u_1<u_2<\ldots<u_n$. Define a $q$-edge-coloring $\xi$ of the complete digraph on $[t]$, where $\xi(i,j)=\psi((u_1-1)t+i,(u_2-1)t+j)$. Since $t=f(q,n)$, there is a monochromatic walk $i_1,\ldots,i_n$ of length $n$ with respect to the coloring $\xi$. Then the vertices $(u_1-1)t+i_1,(u_2-1)t+i_2,\ldots,(u_n-1)t+i_n$ form a monochromatic monotone path in the $q$-edge-coloring $\psi$ of $H$, which completes the proof.
\qed

\section{Lower bounds on $N_k(q,n)$}\label{Section4}

In this section we adapt the stepping-up approach of Erd\H{o}s and Hajnal on hypergraph Ramsey numbers to provide lower bounds for $N_k(q,n)$, stated in Theorems~\ref{main1} and~\ref{main2}. First we address the case $k=3$.

\smallskip

\begin{theorem}\label{stepupto3}
For any integers $q,n\ge 2$, we have $$N_3(q,n) > 2^{f(q,n-1)-1}.$$
\end{theorem}

\medskip
\noindent{\bf Proof.}
Let $\phi$ be a $q$-coloring of the edges of the complete digraph $D$ on vertex set $[0,f(q,n-1)-2]$ without a monochromatic walk on $n-1$ vertices.
We use $\phi$ to define a $q$-coloring $\chi$ of the hyperedges (triples) of the complete ordered $3$-uniform hypergraph $K^3_N$ on the vertex set $V=[N]$ with $N=2^{f(q,n-1)-1}$, as follows.

For any $a \in V$, write $a-1=\sum_{i=0}^{f(n-1,q)-2}a(i)2^i$ with $a(i) \in \{0,1\}$ for each $i$. For $a \not = b$, let $\delta(a,b)$ denote the largest $i$ for which $a(i) \not = b(i)$. Obviously, we have $\delta(a,b) \not = \delta(b,c)$ for every triple $a<b<c$.

Given any triple $a<b<c$, define $\chi(a,b,c)=\phi(\delta(a,b),\delta(b,c))$. To complete the proof of the theorem, it is enough to show that, with respect to this coloring, there is no monochromatic monotone path of length $n$. Suppose for contradiction that there is such a path, and denote its vertices by $a_1<a_2<\ldots<a_n$. Letting $\delta_j=\delta(a_j,a_{j+1})$ for every $j\; (1\le j<n)$, it follows from the definition of the coloring $\chi$ that the integers $\delta_1,\ldots,\delta_{n-1}\in [0,f(q,n-1)-2]$ induce a monochromatic walk in the digraph $D$, with respect to the coloring $\phi$. This contradiction completes the proof.
\qed

\medskip

From Lemma~\ref{lowf}, we have the estimate $f(q,n-1)-1 \geq m^{q-1}$ with $m=1+\lfloor\frac{n-3}{q-1}\rfloor$. It is easy to check that $m \geq n/q$ if
$n \geq q+2$. Together with Theorem~\ref{stepupto3}, we immediately obtain the following corollary, which is the same as the lower bound in Theorem \ref{main1}.

\smallskip

\begin{corollary}\label{corclear}
For any integers $q \ge 2$ and $n \ge q+2$, we have $$N_3(q,n) \geq 2^{(n/q)^{q-1}}.$$
\end{corollary}

We next give a recursive lower bound on $N_k(q,n)$ for $k \geq 4$. The proof is an adaptation of an improved version of the stepping-up technique, due to Conlon, Fox, and Sudakov \cite{CFS10a}. It easy to see that, together with Corollary \ref{corclear}, this gives the lower bound in Theorem \ref{main2}.

\smallskip

\begin{theorem}\label{stepup4}
For any integers $n\ge k \geq 4$ and $q\ge 2$, we have $$N_k(q,n+3) > 2^{N_{k-1}(q,n)-1}.$$
\end{theorem}

\medskip
\noindent{\bf Proof.}
We start the proof in a way similar to Theorem~\ref{stepupto3}.

Let $\phi$ be a $q$-coloring with colors $1,\ldots, q$ of the edges of the complete ordered $(k-1)$-uniform hypergraph on $N_{k-1}(q,n)-1$ vertices without a monochromatic monotone path on $n$ vertices. We use $\phi$ to define a $q$-coloring $\chi$ of the edges of the complete ordered $k$-uniform hypergraph $K^k_N$ on the vertex set $V=[N]$ with $N=2^{N_{k-1}(q,n)-1}$, as follows.

For any $a \in V$, write $a-1=\sum_{i=0}^{N_{k-1}(q,n)-2}a(i)2^i$ with $a(i) \in \{0,1\}$ for each $i$. For $a \not = b$, let $\delta(a,b)$ denote the largest $i$ for which $a(i) \not = b(i)$. As in the proof of Theorem \ref{stepupto3}, we have $\delta(a,b) \not = \delta(b,c)$ for every triple $a<b<c$.

Given any $k$-tuple $a_1<a_2<\ldots<a_k$ of $V$, consider the integers $\delta_i=\delta(a_i,a_{i+1}), 1\le i\le k-1$. If $\delta_1,\ldots,\delta_{k-1}$ form a monotone sequence, then let $\chi(a_1,a_2,\ldots,a_k)=\phi(\delta_1,\delta_2,\ldots,\delta_{k-1})$.

Now we have to color the $k$-tuple $(a_1,\ldots,a_k)$ in the case when $\delta_1,\ldots,\delta_{k-1}$ is not monotone. We say that $i$ is a {\it local minimum} if $\delta_{i-1}>\delta_i<\delta_{i+1}$, a {\it local maximum} if $\delta_{i-1}<\delta_i>\delta_{i+1}$, and a {\it local extremum} if it is either a local minimum or a local maximum. Since $\delta_{i-1} \not = \delta_i$ for every $i$, every nonmonotone sequence $\delta_1,\ldots,\delta_{k-1}$ has a local extremum. Let $i_1$ denote the first local extremum. If $\delta_1,\ldots,\delta_{k-1}$ is not monotone, we define $\chi(a_1,\ldots,a_k)=1$ if $i_1$ is even and a local maximum, or if $i_1$ is odd and a local minimum. Otherwise, let $\chi(a_1,\ldots,a_k)=2$.

\smallskip

Suppose for contradiction that the vertices $a_1<\ldots < a_{n+3}$ induce a monochromatic monotone path with respect to the $q$-coloring $\chi$. For $ 1\leq i \leq n+2$, let $\delta_i=\delta(a_i,a_{i+1})$. Since there is no monochromatic monotone path on $n$ vertices with respect to the coloring $\phi$, the sequence $\delta_1,\ldots,\delta_{n+2}\in[0,N_{k-1}(q,n)-2]$ must have a local extremum. Moreover, if $i_1$ denotes the location of the first local extremum, we have $2 \leq i_1 \leq n-1$.

\medskip

\noindent{\sc Case 1:} The first local extremum $i_1$ satisfies $i_1>2$. We claim that, if $i_1 < k-1$, then the first two edges of the path, $e_1=(a_1,\ldots,a_k)$ and $e_2=(a_2,\ldots,a_{k+1})$, receive different colors, contradicting our assumption that the path is monochromatic. If $i_1 \geq k-1$, then the edges
$e_1=(a_{i_1-k+3},\ldots,a_{i_1+2})$ and $e_2=(a_{i_1-k+4},\ldots,a_{i_1+3})$ receive different colors, which is again a contradiction. Indeed, in either case the {\it type} (maximum or minimum) of the first local extremum is the same for these two edges, but their locations differ by one and hence have different parity, which implies that $e_1$ and $e_2$ receive different colors.

\medskip

\noindent{\sc Case 2:} The first local extremum is $i_1=2$.

\smallskip

{\sc Case 2(a):} $3$ is a local extremum. Note that consecutive extrema have different types, so that the types of $2$ and $3$ are different. Therefore, the first two edges of the path, $e_1=(a_1,\ldots,a_{k})$ and $e_2=(a_2,\ldots,a_{k+1})$, must have different colors, contradicting our assumption that the path is monochromatic. Indeed, for each of these edges, the first local extremum is the second $\delta$, but these extrema are of different types, and hence $e_1$ and $e_2$ receive different colors.

\smallskip

{\sc Case 2(b):} $3$ is not a local extremum. As the sequence $\delta_2,\ldots,\delta_{n+1}$ cannot be monotone, the sequence of $\delta$'s has a second local extremum $i_2$ satisfying $3< i_2 \leq n$. If $i_2<k$, then the edges $e_1=(a_2,\ldots,a_{k+1})$ and $e_2=(a_3,\ldots,a_{k+2})$ have different colors. If $i_2 \geq k$, then the edges $e_1=(a_{i_2-k+3},\ldots,a_{i_2+2})$ and $e_2=(a_{i_2-k+4},\ldots,a_{i_2+3})$ have different colors. Indeed, in either case, the first local extremum for the pair of edges are of the same type, but their locations differ by one and hence have different parity, which implies they have different colors. This completes the proof of the theorem. \qed

\section{Size and Vertex Online Ramsey numbers for monotone paths}\label{sectionsizeandvertex}

We begin this section by estimating vertex online Ramsey numbers, introduced in Subsection~\ref{online} of the Introduction. We then prove an alternative upper bound on $N_k(q,n)$, in terms of the online Ramsey number $V_{k-1}(q,n)$. We end this section by estimating the size Ramsey number $S_2(q,n)$.

\medskip

\subsection{Games}

In this subsection we study vertex online Ramsey numbers. We first relate the online Ramsey game to another game, which we call the {\it $(q,n)$-lattice game} (or {\it lattice game} for short). In this game there are two players, builder and coordinator. During the game, a sequence of (not necessarily distinct) points in the $q$-dimensional grid $\mathbb{Z}^{q}_{>0}$ with positive coordinates is built. After stage $i$, the sequence has length $i$. In stage $i+1$, a new point $p_{i+1}$ is added to the sequence as described below.

In each step of stage $i+1$, builder picks a point $p_j=(x_{1,j},...,x_{q,j})$ with $j \leq i$ already in the sequence. Coordinator then decides a coordinate $k \in [q]$. The point $p_{i+1}$ must satisfy that its $k$th coordinate is greater than the $k$th coordinate of $p_j$, i.e., $x_{k,j}<x_{k,i+1}$. After some steps, builder decides it is time to end this stage, and coordinator picks a point $p_{i+1}$ satisfying the conditions provided by the steps. Note that in stage $1$, as there are no points yet in the sequence, no steps are taken and the only thing that happens is coordinator picks a point $p_1$ in $\mathbb{Z}^q_{>0}$ to begin the sequence. Let $L(q,n)$ be the minimum (total) number of steps needed for builder to guarantee that there is a point in the sequence with a coordinate at least $n$. The following lemma relates these numbers with online Ramsey numbers.

\begin{lemma}\label{equivlattice}
For all $q$ and $n$, we have $$L(q,n)=V_2(q,n).$$
\end{lemma}
\begin{proof} We first prove the bound $L(q,n) \geq V_2(q,n)$. As builder and painter are playing the Ramsey game, builder and coordinator play a corresponding lattice game. We will show that as long as painter guarantees in the Ramsey game there is no monochromatic monotone path of length $n$, coordinator can guarantee in the lattice game that no point in the sequence will have a coordinate at least $n$. Each vertex $v_i$ in the Ramsey game corresponds to a point $p_i$ in the lattice game. To prove the desired result, coordinator will guarantee that the $k$th coordinate of the point $p_i$ in the lattice game is the length of the longest monotone monochromatic path in color $k$ ending at $v_i$ in the Ramsey game. At each step of stage $i$ in the lattice game, builder picks the point $p_j$. Then, in the Ramsey game builder picks the edge $(v_j,v_i)$ to add. The color $k$, $1 \leq k \leq q$, that painter picks to color the edge $(v_j,v_i)$ is the coordinate which coordinator picks so that the $k$th coordinate of $p_i$ is greater than the $k$th coordinate of $p_j$. At the end of this stage, coordinator picks the point $p_i$ such that the $k$th coordinate of $p_i$ is the length of the longest monotone monochromatic path in color $k$ in the Ramsey game ending at $v_i$. This shows that, no matter what builder's strategy is, coordinator can mimic painter's strategy and continue the lattice game for as long as the Ramsey game.

Next we show that $L(q,n) \leq V_2(q,n)$, which would complete the proof. As builder and coordinator are playing the lattice game, builder and painter play a corresponding Ramsey game. We will show that as long as coordinator guarantees in the lattice game that no point in the sequence will have a coordinate at least $n$, painter can guarantee in the Ramsey game there is no monochromatic monotone path of length $n$. Each point $p_i$ in the lattice game will have a corresponding vertex $v_i$ in the Ramsey game.  To prove the desired result, painter will guarantee that for each vertex $v_i$ and each color $k$, the length of the longest monochromatic path in color $k$ ending at $v_i$ is at most the $k$th coordinate of $p_i$.

At each step of stage $i$ in the Ramsey game, the builder picks an edge $(v_j,v_i)$ to add. Then, in the corresponding lattice game, building picks the point $p_j$ to compare $p_i$ with. The coordinate $k$ that coordinator picks so that the $k$th coordinate of $p_i$ is greater than the $k$th coordinate of $p_j$ will be the color $k$ that painter uses to color the edge $(v_j,v_i)$. The length of the longest monochromatic monotone path in color $k$ with last edge $(v_j,v_i)$ is one more than the length of the longest monochromatic monotone path in color $k$ ending at $v_j$, and this length is at most the $k$th coordinate of $p_j$ by construction. Since the $k$th coordinate of $p_i$ is greater than the $k$th coordinate of $p_j$, this guarantees that the length of the longest monochromatic monotone path in color $k$ ending at $v_i$ is at most the $k$th coordinate of $p_i$. This shows that, no matter what builder's strategy is, painter can mimic coordinator's strategy and continue the Ramsey game for as long as the lattice game.
\end{proof}

The next lemma tells us the minimum number of steps required for builder to win the $(q,n)$-lattice game. The lattice $\mathbb{Z}^q$ naturally comes with a
partial order $\prec$, where $p \prec q$ if $p \not = q$ and the $k$th coordinate of $q$ is at least the $k$th coordinate of $p$ for $1 \leq k \leq q$. We will use this order further on.

\begin{lemma}\label{stepslg}
In the $(q,n)$-lattice game, the minimum number of stages (i.e., the number of points in the constructed sequence) builder needs to guarantee a point in the sequence with a coordinate at least $n$ is $(n-1)^q+1$.
\end{lemma}
\begin{proof}
Indeed, in each stage, builder, by picking all previous points at steps, can ensure that the points in the constructed sequence are distinct. Hence, as there are only $(n-1)^q$ points in $\mathbb{Z}^q_{>0}$ with coordinates at most $n-1$, by the pigeonhole principle, the minimum number of stages required for builder is at most $(n-1)^q+1$. Picking a linear extension of the partial order $\prec$ on $[n-1]^q$ described above, coordinator can guarantee that $p_i$ is the $i$th smallest point in the linear extension. In this way, the minimum number of stages required for builder is at least $(n-1)^q+1$, which completes the proof.
\end{proof}

For a finite subset $S \subset \mathbb{Z}^q$ and a point $p \in \mathbb{Z}^q$, define the {\it position of $p$ with respect to $S$} to be the maximum $t$ such that for each $k \in [q]$, there are at least $t$ points $s \in S$ such that the $k$th coordinate of $p$ is at least the $k$th coordinate of $s$.

\begin{lemma}\label{median}
Let $S$ be a finite nonempty subset of $\mathbb{Z}^q$.
\begin{enumerate}
\item There is a point $p \in S$ whose position with respect to $S$ is at least $|S|/q$.
\item If $p \in \mathbb{Z}^q$ is such that there is no $s \in S$ satisfying $s\prec p$, then the position of $p$ with respect to $S$ is at most $\left(1-\frac{1}{q}\right)|S|$.
\end{enumerate}
\end{lemma}
\begin{proof}
We first show the first part. For each $t$, the number of elements of position less than $t$ is less than $tq$. Indeed, for each of the $q$ coordinates, less than $t$ elements of $S$ have position less than $t$ in that coordinate. All other elements have position at least $t$. Setting $t=|S|/q$, there is a point of $S$ of position at least $|S|/q$.

We now show the second part. Suppose the position of $p$ with respect to $S$ is $t$. Delete from $S$ all elements $s$ for which there is $k \in [q]$ such that the $k$th coordinate of $s$ is larger than the $k$th coordinate of $p$. Since $p$ has position $t$, the number of deleted elements is at most $q(|S|-t)$. If $|S|>q(|S|-t)$, then $p$ is at least as large as an element of $S$, a contradiction. Hence, $|S| \leq q(|S|-t)$, and we have $t \leq (1-\frac{1}{q})|S|$, which completes the proof.
\end{proof}

We now prove an upper bound on $L(q,n)$.

\begin{lemma} \label{uplattice}
Let $a(q,n)=1+(q-1)\frac{\log n}{\log (\frac{q}{q-1})}$ and $b(q,n)=(n-1)^q+1$. In the $(q,n)$-lattice game, builder has a winning strategy which uses the minimum possible number of stages, which is $b(q,n)$, such that each stage uses at most $a(q,n)$ steps. In particular, we have $$L(q,n) \leq b(q,n)a(q,n).$$
\end{lemma}
\begin{proof}
By Lemma \ref{stepslg}, the minimum possible number of stages builder needs to win the $(q,n)$-lattice game is indeed $b(q,n)$.

As there are at most $(n-1)^q$ points in $\mathbb{Z}^q$ with positive coordinates at most $n-1$, it suffices for builder to guarantee at each stage $i$, using at most $a(q,n)$ steps, that coordinator picks a point $p_i$ not already in the sequence. To accomplish this, it suffices for builder to guarantee that for each point $p_j$ in the sequence with $j<i$, there is $k \in [q]$, such that the $k$th coordinate of $p_i$ must be greater than the $k$th coordinate of $p_j$.

Suppose we have finished $i-1$ stages, and we are now on stage $i$. Let $M_0$ denote the set of maximal elements in the already constructed sequence of $i-1$ points. We have $|M_0| \leq (n-1)^{q-1}$ as $M_0$ cannot contain two points with the same $q-1$ first coordinates.

In the first step of stage $i$, builder picks a point $p_{j_1} \in M_0$ of position in $M_0$ at least $|M_0|/q$. Such a point exists by the first part of Lemma \ref{median}. Coordinator then picks a coordinate $k_1$. By the definition of position, there are at least $|M_0|/q$ points in $M_0$ whose $k_1$th coordinate is at most the $k_1$th coordinate of $p_{j_1}$. Therefore, point $p_i$ must have $k_1$th coordinate larger than the $k_1$th coordinate of these at least $|M_0|/q$ points in $M_0$. Let $M_1$ denote those points in $M_0$ whose $k_1$th coordinate is larger than the $k_1$th coordinate of $p_{j_1}$, so $|M_1| \leq (1-1/q)|M_0|$.

After step $h$, we have a set $M_h$ with $|M_h| \leq (1-1/q)^h |M_0|$ such that for every point $p_j \in M_0 \setminus M_h$, there is $k \in [q]$ such that the $k$th coordinate of $p_i$ is guaranteed to be larger than the $k$th coordinate of $p_j$. If $M_h$ is nonempty, then builder moves on to step $h+1$ and picks a point $p_{j_{h+1}} \in M_h$ of position in $M_h$ at least $|M_h|/q$. Such a point exists by the first part of Lemma \ref{median}. Coordinator then picks a coordinate $k_{h+1}$. By the definition of position, there are at least $|M_h|/q$ points in $M_h$ whose $k_{h+1}$th coordinate is at most the $k_{h+1}$th coordinate of $p_{j_{h+1}}$. Therefore, point $p_i$ must have $k_{h+1}$th coordinate larger than the $k_{h+1}$th coordinate of these at least $|M_h|/q$ points in $M_h$. Let $M_{h+1}$ denote those points in $M_h$ whose $k_{h+1}$th coordinate is larger than the $k_{h+1}$th coordinate of $p_{j_{h+1}}$, so $|M_{h+1}| \leq (1-1/q)|M_h|$.

If $M_h$ is empty, then builder decides to end the stage and coordinator must pick a point $p_i$ satisfying the required conditions. Note that we eventually will have an empty $M_h$ as $M_0$ is a finite set, and $M_{h}$ is a proper subset of $M_{h-1}$.  Since in the previous step $M_{h-1}$ is nonempty, we have $$1 \leq |M_{h-1}| \leq (1-1/q)^{h-1}|M_0| \leq (1-1/q)^{h-1}(n-1)^{q-1},$$ and so the number $h$ of steps satisfies  $$h \leq 1+(q-1)\frac{\log n}{\log (\frac{q}{q-1})}=a(q,n).$$ Since $M_0$ are the maximal elements of the already constructed sequence, and for every point $p \in M_0$, there is $k \in [q]$ such that the $k$th coordinate of $p_i$ is greater than the $k$th coordinate of $p$, then $p_i$ cannot be an element of the already constructed sequence. This completes the proof.\end{proof}

We remark that one can do a little better in the above proof by improving the bound on $|M_0|$ by using a result of de Bruijn, van Ebbenhorst Tengbergen, and Kruyswijk \cite{BrEbKr} which extends Sperner's theorem. This result says that the set of elements in $[n-1]^q$ whose coordinates sum to $\lfloor nq/2 \rfloor$ forms a maximum antichain in the poset on $[n-1]^q$. Since $M_0$ is an antichain, the size of this set is an upper bound on $|M_0|$.

We next present a strategy for coordinator which gives a lower bound for $L(q,n)$.

\begin{lemma}\label{lowlattice}
For each fixed $q \geq 2$, we have $L(q,n) \geq \left(q-1-o(1)\right)n^q\log_q n$.
\end{lemma}
\begin{proof}
We present a strategy for coordinator which satisfies that the number of steps in the stage the $i$th point not already in the sequence is added is at least a certain number depending on $i$ which will be defined later in the proof. Level $r$ of the grid $[n-1]^q$ consists of all points in this grid whose coordinates sum to $r$.
The new points will be added in levels, so no points of larger level are added until all points of lower level are already in the sequence.

At stage $i$, for some $r$, all points in levels $j$ with $j<r$ are already in the sequence, and not all of level $r$ is in the sequence. Let $S_0$ denote the set of points in level $r$ not already in the sequence. As $S_0$ is disjoint from the already constructed sequence, and no element in the already constructed sequence is larger than any point in $S_0$, then no matter what point $p_{j_1}$ in the already constructed sequence builder chooses for the first step to compare $p_i$ with, by the second part of Lemma \ref{median}, the position of $p_{j_1}$ with respect to $S_0$ is at most $(1-\frac{1}{q})|S_0|$. Hence, there is $k \in [q]$ such that at least $\frac{1}{q}|S_0|$ points in $S_0$ have $k$th coordinate more than the $k$th coordinate of $p_{j_1}$. Coordinator picks this coordinate $k$ to compare with, and lets $S_1$ denote the points in $S_0$ whose $k$th coordinate is larger than the $k$th coordinate of $p_{j_1}$, so $|S_1| \geq |S_0|/q$.

After $h$ steps of stage $i$, we have a subset $S_h$ of level $r$ with $|S_h| \geq |S_0|/q^h$ and every point in $S_h$ satisfies the conditions that must be satisfied by $p_i$. Notice this is satisfied for $h=1$. Builder has two options, either continue on to step $h+1$, or declare the turn over.

In the first case, builder picks a point $p_{j_{h+1}}$ in the already constructed sequence to compare $p_i$ with. By the second part of Lemma \ref{median}, the position of $p_{j_{h+1}}$ with respect to $S_h$ is at most $(1-\frac{1}{q})|S_h|$. Hence, there is $k \in [q]$ such that at least $\frac{1}{q}|S_h|$ points in $S_h$ have $k$th coordinate more than the $k$th coordinate of $p_{j_{h+1}}$. Coordinator picks this coordinate $k$ to compare with, and lets $S_{h+1}$ denote the points in $S_h$ whose $k$th coordinate is larger than the $k$th coordinate of $p_{j_{h+1}}$, so $|S_{h+1}| \geq |S_h|/q$. This finishes step $h+1$.

In the second case, builder declares the turn is over after the $h$ steps. There are two possibilities: $|S_h|=1$ or $|S_h| > 1$. If $|S_h|=1$, coordinator picks $p_i$ to be the point in $S_h$. The number $h$ of steps in this case satisfies $1 = |S_h| \geq |S_0|/q^h$, so $h \geq \log_q |S_0|$. Otherwise, $|S_h|>1$, and we claim that there is still a point $z \in [n-1]^q$ in a level less than $r$ which can be chosen for $p_i$. Indeed, if $x=(x_1,\ldots,x_q)$ and $y=(y_1,\ldots,y_q)$ are distinct points in $S_h$, then the point $z=(z_1,\ldots,z_q)$ with $z_j=\min(x_j,y_j)$ satisfies the desired properties. Indeed, as $x$ and $y$ are in level $r$, the level of $z$ is $\sum_{i=1}^q \min(x_i,y_i) < \sum_{i=1}^q x_i=r$. Further, the set of possibilities for $p_i$ is of the form that the $k$th coordinate is at least some $m_k$ for each $k$, and, as $x$ and $y$ are in this set, then $z$ is also in this set. The point $z$ is already in the sequence, and coordinator lets $p_i$ be this point. Since $p_i$ is already in the sequence, after stage $i$ builder is exactly in the same situation as before stage $i$ in regards to winning the $(q,n)$-lattice game, and so the steps in this stage did not help get closer to ending the game.

In each stage for which a new point appears in the sequence at least $\log_q |S_0|$ steps were used, where $|S_0|$ is the number of points not yet in the sequence in the level currently being used to fill the new points. Let $l_r$ denote the number of points in level $r$. We have $l_r=1$ if and only if
$r=q$ or $q(n-1)$. We thus have
\begin{eqnarray*}L(q,n) & \geq & \sum_{r=q}^{q(n-1)}\sum_{m=1}^{l_r}\log_q m = \sum_{r=q+1}^{q(n-1)-1}\sum_{m=1}^{l_r}\log_q m \geq \sum_{r=q+1}^{q(n-1)-1}\int_{1}^{l_r-1} \log_q m \, dm \\ & = & \sum_{r=q+1}^{q(n-1)-1}\left(m\log_q m - m/\ln q\right) \Big |_{m=1}^{l_r-1}=  \sum_{r=q+1}^{q(n-1)-1}(l_r-1)\log_q (l_r-1)-(l_r-2)/\ln q \\ & \geq & \left((q(n-2)-1)\ell \log \ell -\left((n-1)^{q}-2-2(q(n-2)-1\right)/\ln q\right)  ,\end{eqnarray*}
where in the second inequality we bounded a sum from below by an integral using the fact $f(x)=\log x$ is an increasing function, and the last inequality follows from Jensen's inequality and convexity of the function $f(x)=x\log x$, where we substitute $\sum_{r=q+1}^{q(n-1)-1}\ell_r =(n-1)^q-2$, and  $\ell=\left((n-1)^q-1-q(n-2)\right)/\left(q(n-2)-1\right)$ is the average of $\ell_r-1$ with $r$ ranging from $q+1$ to $q(n-1)-1$.
For $q$ fixed, this bound can be simplified to
$$L(q,n) \geq (1-o(1))n^{q}\log_q n^{q-1}=
\left(q-1-o(1)\right)n^q\log_q n,$$
which completes the proof.
\end{proof}

\subsection{Alternative upper bound on $N_k(q,n)$}

In this subsection, we present the proof of Theorem \ref{alternativeupperbound}, which gives an upper bound on $N_k(q,n)$ in terms of the online Ramsey number $V_{k-1}(q,n+k-2)$.

Define $V'_k(q,n)$ exactly as $V_k(q,n)$, except in each stage $t \geq k$, builder must add at least one edge (whose largest vertex is $v_t$). We refer to the online Ramsey game in which builder must add at least one edge in each stage $t \geq k$ as the {\it modified online Ramsey game}. We have \begin{equation}\label{modifiedonline}
V_k(q,n) \leq V'_k(q,n) \leq V_k(q,n+k-1),\end{equation} the lower bound being trivial. Note that in the usual online Ramsey graph,  those vertices $v_t$ which are not the largest vertex in an edge can only be among the first $k-1$ vertices of a monotone path. Hence, after a monochromatic monotone path of length $n+k-1$ is created, the last $n$ vertices of the path each are the largest vertex of some edge, which gives the upper bound. We will use this upper bound in the proof below.

\medskip

\noindent\textbf{Proof of Theorem \ref{alternativeupperbound}.}  Let $\chi: {[N]\choose k}\rightarrow[q]$ be a $q$-coloring of the edges of $K^k_N$, where $N = q^{V_{k-1}(q,n+k-2)} + k-2$.  We have to show that $K^k_N$ contains a monochromatic monotone path of length $n$.

We construct a set of vertices $\{v_{1},v_{2},...,v_{t}\}$ and a $(k-1)$-uniform hypergraph $H$ on these vertices with at most $V_{k-1}(q,n+k-2)$ edges such that for any $(k-1)$-edge $e = \{v_{i_1},...,v_{i_{k-1}}\}$, $i_1 < \cdots <i_{k-1}$, the color of any $k$-edge $\{v_{i_1},...,v_{i_{k-1}},v_{i_k}\}$ in $K^k_N$ with $i_k > i_{k-1}$ is the same, say $\chi'(e)$.  Moreover, this $(k-1)$-uniform hypergraph will contain a monochromatic path of length $n$, which one can easily see will define a monochromatic path of length $n$ in $K^k_N$. As described below, the way the vertices of the set are constructed will be determined by playing the modified online Ramsey game.

We begin the construction of this set of vertices by setting $v_1 = 1$ and setting $S_{1} = [N]\setminus \{v_1\}$. After stage $t$ of the process, we have a set of vertices $\{v_1,...,v_t\}$ with $v_1<\cdots<v_t$ and a set $S_t$ such that $w>v_t$ for all $w \in S_t$ and
for each $(k-1)$-edge $\{v_{i_1},...,v_{i_{k-1}}\}$ in $H$ with $i_1<\cdots<i_{k-1} \leq t$, the color of the $k$-edge $\{v_{i_1},...,v_{i_{k-1}}, w\}$ in $K^k_N$ is the same for every $w$ in $S_t$ and for every $w=v_j$ with $i_{k-1}<j \leq t$.

In the beginning of stage $t+1$, we let $v_{t+1}$ be the smallest element in $S_t$. We play the modified online Ramsey game, so that builder chooses the edges, which are $(k-1)$-tuples $\{v_{i_1},\ldots,v_{i_{k-2}},v_{t+1}\}$ with $i_1<\cdots<i_{k-2}<t+1$, to be drawn according to his strategy.  Painter then colors these edges.  For the first edge $e_1$ chosen, painter looks at all $k$-tuples containing this edge and a vertex from $S_t\setminus \{v_{t+1}\}$.  The $(k-1)$-edge is colored $r_1 \in [q]$ in $\chi'$ if there are at least $(|S_{t}| - 1)/q$ such $k$-tuples that have color $r_1$, breaking ties between colors arbitrarily. Such a color $r_1$ exists by the pigeonhole principle. This defines a new subset $S_{t,1}$, which are all vertices in $S_{t}\setminus \{v_{t+1}\}$ such that together with edge $e_1$ form a $k$-tuple of color $r_1 = \chi'(e_1)$. After $j$ edges have been drawn in this stage, we have a subset $S_{t,j}$, and we color the next drawn edge $e_{j+1}$ color $r_{j+1}$ if the number of $k$-tuples of color $r_{j+1}$ containing it and a vertex from $S_{t,j}$ is at least $|S_{t,j}|/q$, breaking ties between colors arbitrarily. Then we define $S_{t,j+1}$ to be the set of at least $|S_{t,j}|/q$ vertices in $S_{t,j}$ which together with $e_{j+1}$ form an edge of color $r_{j+1}$.

After builder has added all $(k-1)$-edges with largest vertex $v_{t+1}$, the remaining set will be $S_{t+1}$.  Let $m_t$ be the number of $(k-1)$-edges $e = \{v_{i_1},...,v_{i_{k-2}}, v_t\}$ with $i_1<\cdots<i_{k-2} < t$ in $H$. Since we are playing the modified online Ramsey game, we have $m_t=0$ for $1 \leq t \leq k-2$ and $m_t \geq 1$ for $t \geq k-1$.

We now show by induction that

$$|S_t| \geq \frac{N - (k-2)}{q^{\sum_{i = k-1}^t m_i}} - \sum\limits_{i = k-1}^{t} \frac{1}{q^{\sum_{j=i}^t m_j}}.$$

\noindent For the base case $t = k-2$, we have

$$|S_{k-2}| = N - (k-2).$$

Suppose we have proved the desired inequality for $t$.  When we draw a vertex $v_{t+1}$, the size of our set decreases by 1.  Each time we draw an edge from $v_{t+1}$, the size of our set $S$ goes down by a factor of $q$.  Therefore

\begin{eqnarray*}|S_{t+1}| & \geq & \frac{|S_t|-1}{q^{m_{t+1}}} \geq  \frac{N - (k-2)}{q^{\sum_{i = k-1}^{t+1} m_i}} - \left(\sum\limits_{i = k-1}^{t} \frac{1}{q^{\sum_{j=i}^{t+1} m_j}}\right)  - \frac{1}{q^{m_{t+1}}} \\
      & = & \frac{N - (k-2)}{q^{\sum_{i = k-1}^{t+1} m_i}} - \sum\limits_{i = k-1}^{t+1} \frac{1}{q^{\sum_{j=i}^{t+1} m_j}}.\end{eqnarray*}

The number of edges drawn in the $(k-1)$-uniform hypergraph after $t$ stages is $\sum_{i=k-1}^t m_i$. If this is at most $V'_{k-1}(q,n) \leq V_{k-1}(q,n+k-2)$,
as $q \geq 2$ and $m_i \geq 1$ for all $i \geq k-1$, we have

\begin{eqnarray*} |S_{t}| & \geq & \frac{N - (k-2)}{q^{\sum_{i = k-1}^t m_i}} - \sum\limits_{i = k-1}^{t} \frac{1}{q^{\sum_{j=i}^t m_j}} \geq \frac{N - (k-2)}{q^{V'_{k-1}(q,n)}} - \sum\limits_{i = k-1}^{t} \frac{1}{q^{t-i + 1}} \\      & > &\frac{N - (k-2)}{q^{V'_{k-1}(q,n)}} - 1 = 0.\end{eqnarray*}

Since $|S_t|$ is an integer, $|S_t| \geq 1$ and we can continue to the next stage. Thus, at the time we stop, $\sum_{i=1}^t m_i > V'_{k-1}(q,n)$ and we have constructed a $(k-1)$-uniform hypergraph $H$ that contains a monochromatic monotone path in coloring $\chi'$ on vertices $v_{i_1},\ldots,v_{i_n}$.   On the other hand, it follows from the properties of our construction that the vertices $v_{i_1},\ldots,v_{i_n}$ form a monochromatic monotone path of length $n$ in $K^k_N$ as required.
\qed

\medskip

\subsection{Size Ramsey numbers}

We close this section by proving a lower bound on the size Ramsey number $S_2(q,n)$. Recall that $\chi(q,n)$ denotes the minimum chromatic number of an ordered graph $G$ which is $(q,n)$-path Ramsey. A graph is {\it $t$-degenerate} if every subgraph of it has a vertex of degree at most $t$. Every $t$-degenerate graph has an ordering of its vertices such that each vertex is adjacent to at most $t$ earlier vertices. Indeed, let a vertex of degree at most $t$ be the last vertex in the ordering, and continue picking vertices out of degree at most $t$ from the remaining induced subgraph, adding them to the end of the ordering, until all vertices are picked out. From  this ordering, we get that every $t$-degenerate graph has chromatic number at most $t+1$ as we can color the vertices in order, picking the color of a vertex distinct the colors of its at most $t$ neighbors that come before it in the ordering.

\smallskip

\begin{theorem}
Let $n_1,n_2,q \geq 2$ be integers, and let $n=n_1+n_2-1$. We have $$S_2(q,n) \geq (\chi(q,n_1)-1)N_2(q,n_2)/2.$$
\end{theorem}

\medskip
\noindent{\bf Proof.}
Let $G$ be an ordered graph such that for every $q$-edge-coloring of $G$ there is a monochromatic monotone path of length $n$. Suppose for contradiction that $G$ has fewer than $(\chi(q,n_1)-1)N_2(q,n_2)/2$ edges. Partition $V(G)=V_1 \cup V_2$, where $V_2$ is the $t$-core of $G$ with $t=\chi(q,n_1)-1$, which is formed by deleting, one-by-one, vertices from $G$ of smallest degree until the remaining induced subgraph has minimum degree at least $t$. Since $G$ has fewer than
$(\chi(q_1,n)-1)N_2(q,n_2)/2$ edges, and the minimum degree in the subgraph of $G$ induced by the vertex set $V_2$ is at least $t$, we have the inequality $|V_2|<N_2(q,n_2)$, and hence there is a $q$-coloring with color set $[q]$ of the edges inside $V_2$ with no monochromatic monotone path of length $n_2$. The subgraph of $G$ induced by the vertices in $V_1$ is $(t-1)$-degenerate, so it can be properly colored with $t=\chi(q,n_1)-1$ colors. Hence, by definition of $\chi(q,n_1)$, there is a $q$-edge-coloring with color set $[q]$ of the edges inside $V_1$ without a monochromatic monotone path of length $n_1$. We still have to color the edges between $V_1$ and $V_2$. For any two vertices $i \in V_1$ and $j \in V_2$, if $i<j$, color the edge $(i,j)$ color $1$, and if $i>j$, color the edge $(i,j)$ color $2$. For each $j \in \{1,2\}$, a monochromatic monotone path cannot start with a vertex in $V_j$, later have a vertex in $V_{3-j}$, and then, even later have a vertex in $V_j$, as otherwise it would contain an edge of color $1$ and an edge of color $2$. Therefore, the length of the longest monochromatic monotone path is at most the sum of the lengths of the longest monochromatic monotone paths within $V_1$ and within $V_2$, which is at most $(n_1-1)+(n_2-1)=n-1$, contradicting our assumption that every $q$-edge-coloring of $G$ has a monochromatic monotone path of length $n$. The proof is complete.
\qed

\medskip

Taking $n_1=n_2=n/2$ in the previous theorem, and using $N_2(q,n)=(n-1)^q+1$ and Theorems \ref{ftheorem} and \ref{chif}, we have the following corollary.

\begin{corollary}
For any integers $n, q\ge 2$, we have $S_2(q,n) \geq c_qn^{2q-1}$ for some constant $c_q>0$ only depending on $q$.
\end{corollary}

This should be compared with the trivial upper bound $$S_2(q,n) \leq {N_2(q,n) \choose 2} = {(n-1)^q+1 \choose 2} \leq n^{2q}.$$
Thus, for fixed $q$, the lower and upper bound on $S_2(q,n)$ are roughly a factor $n$ apart.

\section{Transitive colorings}\label{transitivesect}
In this short section we prove a simple lemma which is used in the proof of Theorem \ref{convex}. A {\em family} $\mathcal{F}$ of $k$-element subsets of $[N]$ (i.e., a $k$-uniform hypergraph on the vertex set $[N]$) is said to be {\em transitive} if for any $i_1<i_2<\ldots<i_{k+1}$ such that
$$\{i_1,i_2,\ldots, i_{k}\}, \{i_2, i_3,\ldots, i_{k + 1}\}\in \mathcal{F},$$
we can conclude that {\em all} $k$-element subsets of $\{i_1, i_2,\ldots, i_{k+1}\}$ belong to $\mathcal{F}$. A {\em $q$-coloring} of all $k$-element subsets of $[N]$ is called {\em transitive} if each of its color classes is transitive.

For the geometric application Theorem \ref{convex}, we need the following lemma relating monochromatic cliques in transitive colorings to monochromatic monotone paths in general colorings.

\begin{lemma}
\label{clique}
Let $N=N_k(q,n)$ be the smallest integer such that for every $q$-coloring of all hyperedges of $K^k_N$, the complete $k$-uniform hypergraph on the vertex set $[N]$, there exists a monochromatic monotone path of length $n$. Then, for every transitive $q$-coloring of all hyperedges of $K^k_N$, there exists a monochromatic complete subhypergraph of size $n$.
\end{lemma}

Lemma \ref{clique} directly follows from the definition of $N_k(q,n)$  and the following statement.

\begin{lemma}\label{complete}
Let $n > k$, and let $H$ be a $k$-uniform ordered hypergraph on the vertex set $[n]$, which contains a monotone path of length $n$, that is, $\{i,i+1,\ldots,i+k-1\}\in E(H)$ for all $1\le i\le n-k+1$.

If $E(H)$ is transitive, then $H$ is the complete $k$-uniform hypergraph on $[n]$.
\end{lemma}
\begin{proof}
We proceed by induction on $n$.  The base case $n = k + 1$ follows from the definition of transitivity.

For the inductive step, suppose that $n>k+1$ and that the statement has already been established for all ordered hypergraphs with fewer than $n$ vertices. Let $H$ be a hypergraph with $n$ vertices meeting the conditions of the lemma. By the induction hypothesis, the induced subhypergraph $H[\{1,\ldots,n-1\}]\subseteq H$, consisting of all edges of $H$ contained in $\{1,\ldots,n-1\}$ is complete, and so is the induced subhypergraph $H[\{2,\ldots,n\}]\subseteq H$.

Observe that this implies that for any $1\le i\le n$, the induced hypergraph $H[\{1,\ldots,n\}\setminus\{i\}]$ contains a monotone path of length $n-1$. As the edge set of $H[\{1,\ldots,n\}\setminus\{i\}]$ is clearly transitive, we can apply the induction hypothesis to conclude that $H[\{1,\ldots,n\}\setminus\{i\}]$ is a complete $k$-uniform hypergraph, for all $i$. Therefore, $H$ is also a complete $k$-uniform hypergraph.\end{proof}

\section{Noncrossing convex bodies---Proof of Theorem~\ref{convex}}\label{geomsect}

Let $\mathcal{C} = \{C_1,C_2,...,C_N\}$ be a family of $N = N_3(3,n)$ noncrossing convex bodies in general position in the plane, numbered from left to right, according to the $x$-coordinates of their left endpoints (leftmost points). We have to show that $\mathcal{C}$ contains $n$ members in convex position.

In order to complete the proof outlined at the end of Subsection~\ref{noncrossing} in the Introduction, we have to define a $3$-coloring of all triples $(C_i,C_j,C_k)$ with $i<j<k$, which satisfies some special properties. We need some definitions.

We say that the triple $(C_i,C_j,C_k), i<j<k$ has a \emph{clockwise (counterclockwise) orientation} if there exist distinct points $q_i\in C_i, q_j\in C_j, q_k\in C_k$ such that they lie on the boundary of $conv(C_i\cup C_j\cup C_k)$ and appear there in clockwise (counterclockwise) order (cf. \cite{pt3}, \cite{suk}). Clearly, the orientation of a triple can be clockwise and counterclockwise at the same time. See Figure 1.

 \begin{figure}[h]
  \centering
    \subfigure[$(C_i,C_j,C_k)$ not in general position.]{\label{hide2}\includegraphics[width=.2\textwidth]{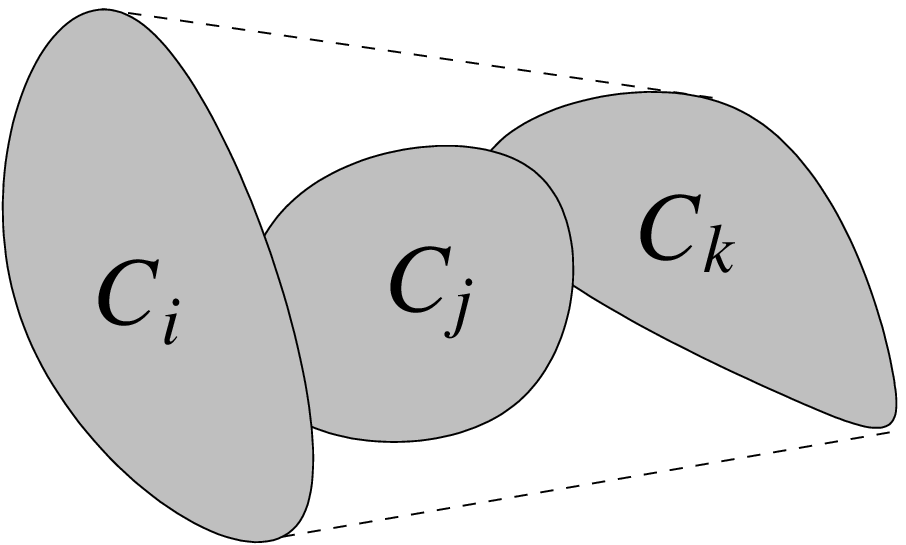}}  \hspace{1cm}
      \subfigure[Clockwise orientation.]{\label{clockwise}\includegraphics[width=0.2\textwidth]{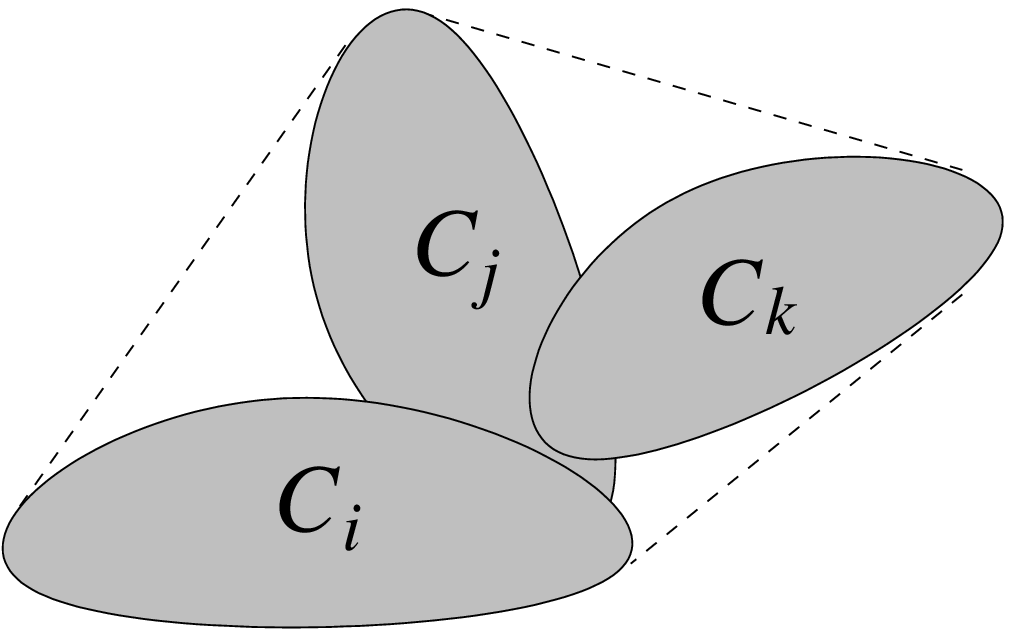}}   \hspace{1cm}
  \subfigure[Counterclockwise orientation.]{\label{counter}\includegraphics[width=.16 \textwidth]{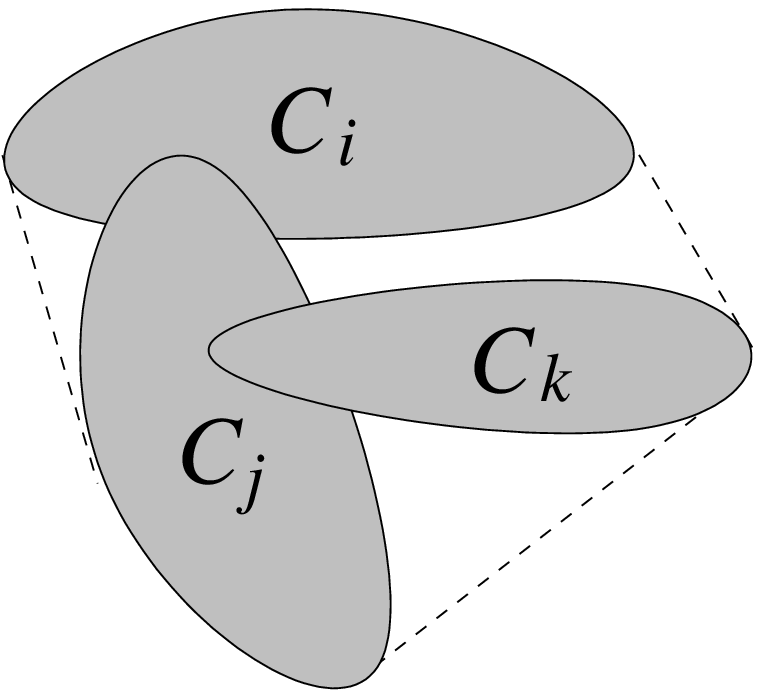}}  \hspace{1cm}
    \subfigure[Both orientations.]{\label{both}\includegraphics[width=0.16 \textwidth]{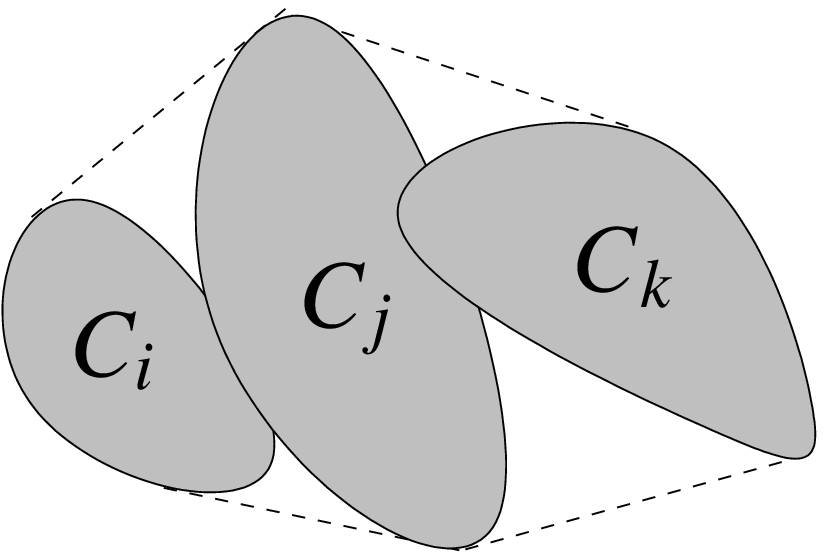}}
\caption{Orientations of convex bodies.}
  \label{fig:order}
\end{figure}

As we numbered the members of $\mathcal{C}$ according to the order of their left endpoints, for any $i < j< k$, the left endpoint of $C_i$ must lie on the boundary of $conv(C_i\cup C_j\cup C_k)$. The triple $(C_i,C_j,C_k)$ is said to have a {\em strong-clockwise (strong-counterclockwise) orientation} if there exist points $q_j\in C_j, q_k\in C_k$ such that, starting at the left endpoint $q_i^*$ of $C_i$, the triple $(q_i^*,q_j,q_k)$ appears in clockwise (counterclockwise) order along the boundary of $conv(C_i\cup C_j\cup C_k)$.  We say that $(C_i,C_j,C_k)$ has \emph{both strong orientations} if it has both a strong-clockwise and a strong-counterclockwise orientation. Notice that it is possible that $(C_i,C_j,C_k)$ has both orientations, but only one \emph{strong} orientation. On the other hand, every triple has at least one strong orientation. Obviously, if $(C_i,C_j,C_k)$ has a strong-clockwise orientation, say, then it also has a clockwise orientation.

\subsection{Definitions, notation, and observations}

Let $\mathcal{C} = \{C_1,C_2,...,C_N\}$ be a family of noncrossing convex bodies in general position, numbered from left to right, according to the $x$-coordinates of their left endpoints (leftmost points).  A subfamily of convex bodies $\mathcal{S}$ is called {\em separable} if it has a member $C$ such that $conv(\cup\mathcal{S})\setminus C$ is disconnected.  We say that $C$ \emph{separates} $\mathcal{S}$ if $conv(\cup\mathcal{S})\setminus C$ is disconnected.  Clearly if a triple is not separable, then it only has one strong orientation.  Since $\mathcal{C}$ is a family of noncrossing convex sets, every triple has at most one member that separates it.

We will make several observations on the triples of $\mathcal{C}$ using the following notation.  For $i < j < k$, let $\beta$ denote the boundary of $conv(C_i\cup C_j\cup C_k)$, $q_i^*$ denote the left endpoint of $C_i$, and let $\gamma_i = C_i\cap \beta, \gamma_j = C_j\cap \beta, \gamma_k = C_k\cap \beta$.

  \begin{observation}
\label{cknotseparate}

For $i < j < k$, $C_k$ cannot separate $(C_i,C_j,C_k)$.

\end{observation}

\noindent \textbf{Proof.}  For sake of contradiction, suppose that $C_k$ separates the triple $(C_i,C_j,C_k)$.   Then $\gamma_i$ is a simple continuous curve (i.e., homeomorphic to the unit interval) that contains the left most point of $\beta$, $\gamma_j$ is also a simple continuous curves, and $\gamma_k$ consists of two disjoint simple continuous curves $\gamma_{k_1}$ and $\gamma_{k_2}$.  Since no two convex bodies have a common tangent that meets at their intersection, $\gamma_i,\gamma_j,\gamma_{k_1},\gamma_{k_2}$ are pairwise disjoint.  Moreover, $\gamma_i,\gamma_{k_1},\gamma_j,\gamma_{k_2}$ appear in clockwise order along $\beta$. However since $q_i^* \in \gamma_i$ and since the left endpoint of $C_j$ lies inside $conv(C_i,C_j,C_k)$ to the left of the left endpoint of $C_k$, this implies that $C_j$ and $C_k$ cross, that is, they share more than two boundary points which is a contradiction.

$\hfill\square$

\begin{observation}
\label{bothstrongcj}
For $i < j < k$, the triple $(C_i,C_j,C_k)$ has both strong orientations if and only if $C_j$ separates it.

\end{observation}

\noindent \textbf{Proof.}  If $C_j$ separates the triple $(C_i,C_j,C_k)$, then $\gamma_i$ and $\gamma_k$ are simple continuous curves such that $q_i^* \in \gamma_i$, while $\gamma_j$ consists of two disjoint simple continuous curves $\gamma_{j_1},\gamma_{j_2}$.  Moreover since $\mathcal{C}$ is in general position, $\gamma_i,\gamma_{j_1},\gamma_{k},\gamma_{j_2}$ are pairwise disjoint and appear in clockwise order along $\beta$.  Therefore we can find points $q_{j_1}\in \gamma_{j_1}, q_{j_2}\in\gamma_{j_2}$, and $q_k \in \gamma_k$ such that $q^*_i,q_{j_1},q_k,q_{j_2}$ appear on $\beta$ in clockwise order.  Hence $(C_i,C_j,C_k)$ has both strong orientations.

For the other direction, suppose that $(C_i,C_j,C_k)$ has both strong orientations.  Since the triple $(C_i,C_j,C_k)$ is separable, and we know that $C_k$ does not separate it, it suffices to show that $C_i$ does not separate $(C_i,C_j,C_k)$.

Suppose $C_i$ separates $(C_i,C_j,C_k)$.   Then $\gamma_j$ and $\gamma_k$ are simple continuous curves, and $\gamma_i$ consists of two disjoint simple continuous curves $\gamma_{i_1}$ and $\gamma_{i_2}$.  We can assume that $q_i^* \in \gamma_{i_1}$.  Since $C_i$ separates $(C_i,C_j,C_k)$, $\gamma_{i_1},\gamma_j,\gamma_{i_2},\gamma_k$ are pairwise disjoint and must appear in clockwise or counterclockwise order along $\beta$.  If $\gamma_{i_1},\gamma_j,\gamma_{i_2},\gamma_k$ appear in clockwise order, then $(C_i,C_j,C_k)$ has only a strong-clockwise orientation since there does not exists points $q_j\in \gamma_j$ and $q_k\in \gamma_k$ such that $q_i^*,q_j,q_k$ appear in counterclockwise order along $\beta$.  Likewise if $\gamma_{i_1},\gamma_j,\gamma_{i_2},\gamma_k$ appear in counterclockwise order along $\beta$, then $(C_i,C_j,C_k)$ has only a strong-counterclockwise orientation.  In either case, $(C_i,C_j,C_k)$ does not have both strong orientations and we have a contradiction.

$\hfill\square$

Suppose that $C_i$ separates $(C_i,C_j,C_k)$.  By the argument above $(C_i,C_j,C_k)$ has only one strong orientation since $\gamma_{i_1},\gamma_{j},\gamma_{i_2},\gamma_k$ must appear in clockwise or counterclockwise order along $\beta$.   Let $h$ be a fixed line that goes through $q^*_i$ and intersects $\gamma_{i_2}$.  We can assume that $h$ is not vertical.  Then by the noncrossing property and by convexity of $conv(C_i\cup C_j\cup C_k)$, regions $C_j\setminus C_i$ and $C_k\setminus C_i$ must be separated by $h$.  Therefore we make the following observation.

\begin{observation}
\label{ciseparates}
For $i < j < k$, suppose that $C_i$ separates the triple $(C_i,C_j,C_k)$ and let $h$ be the line described as above.  If $C_j\setminus C_i$ lies in the upper half-plane generated by $h$ (and $C_k\setminus C_i$ lies in the lower half-plane), then $(C_i,C_j,C_k)$ has only a strong-clockwise orientation.  Likewise if $C_j\setminus C_i$ lies in the lower half-plane generated by $h$ (and $C_k\setminus C_i$ lies in the upper half-plane), then $(C_i,C_j,C_k)$ has only a strong-counterclockwise orientation.

\end{observation}

\noindent \textbf{Proof.}  If $C_j\setminus C_i$ lies in the upper half-plane generated by $h$, then $\gamma_{i_1},\gamma_j, \gamma_{i_2},\gamma_k$ must appear in clockwise order along $\beta$.  Hence there does not exists points $q_j \in \gamma_j$ and $q_k\in \gamma_k$ such that $q^*_i,q_j,q_k$ appear in counterclockwise order along $\beta$.  Therefore $(C_i,C_j,C_k)$ does not have a strong-counterclockwise orientation, which implies that it only has a strong-clockwise orientation since every triple has at least one strong orientation.

By a similar argument, if $C_j\setminus C_i$ lies in the lower half-plane generated by $h$, this implies that $(C_i,C_j,C_k)$ has only a strong-counterclockwise orientation.

$\hfill\square$

For each pair of indices $i < j$, let $\vec{l}_{ij}$ (and $\vec{r}_{ij}$) denote the \emph{directed} common tangent line of $C_i$ and $C_j$ which meets $C_i$ before $C_j$, such that $C_i\cup C_j$ lies completely to the right (left) of $\vec{l}_{ij}$ ($\vec{r}_{ij}$).

Let $\mathcal{C}$ be a family of noncrossing convex sets in general position, and let $\beta$ denote the boundary of $conv(\cup \mathcal{C})$.  We say that the pair $(C_i,C_j)$ appears in \emph{consecutive clockwise (counterclockwise) order} on $\beta$ if there exists points $q_i\in C_i\cap \beta$ and $q_j\in C_j\cap \beta$ such that if $\gamma$ is the arc generated by moving from $q_i$ to $q_j$ along $\beta$ in the clockwise (counterclockwise) direction, then the interior of $\gamma$ does not intersect any member of $\mathcal{C}$.

If $(C_i,C_j)$ appears in consecutive clockwise (counterclockwise) order on $\beta$, then $\vec{l}_{ij}$ ($\vec{r}_{ij}$) is the directed line going through points $q_i$ and $q_j$ (where $q_i,q_j$ is defined as above) such that $conv(\cup \mathcal{C})$ lies to the right (left) of $\vec{l}_{ij}$ ($\vec{r}_{ij}$).  Therefore we have the following.

\begin{observation}
\label{bothstrongvec}
For $i < j< k$, if $(C_i,C_j,C_k)$ has both strong orientations, then $C_i\cup C_j\cup C_k$ lies to the right of $\vec{l}_{ij}$ and $\vec{l}_{jk}$, and lies to the left of $\vec{r}_{ij}$ and $\vec{r}_{jk}$. In other words, $(C_i,C_j)$ (and $(C_j,C_k)$) appears in consecutive clockwise and counterclockwise order on the boundary of $conv(C_i,C_j,C_k)$.

\end{observation}

\noindent \textbf{Proof.} By Observation \ref{bothstrongcj}, $C_j$ separates $(C_i,C_j,C_k)$.  Therefore $\gamma_i$ and $\gamma_k$ are simple continuous curves while $\gamma_j$ consists of two disjoint simple continuous curves $\gamma_{j_1}$ and $\gamma_{j_2}$.  Moreover, $\gamma_i,\gamma_{j_1},\gamma_k,\gamma_{j_2}$ are pairwise disjoint and appear in clockwise order along the boundary of $conv(C_i\cup C_j\cup C_k)$.  Hence $(C_i,C_j)$ (and $(C_j,C_k)$) appears in consecutive clockwise and counterclockwise order on the boundary of $conv(C_i,C_j,C_k)$.
$\hfill\square$

\begin{observation}
\label{clockwisevec}
For $i < j< k$, if $(C_i,C_j,C_k)$ has only a strong-clockwise orientation, then $C_i\cup C_j\cup C_k$ lies to the right of $\vec{l}_{ij}$ and to the left of $\vec{r}_{ik}$.  In other words, $(C_i,C_j)$ appears in consecutive clockwise order and $(C_i,C_k)$ appears in consecutive counterclockwise order on the boundary of $conv(C_i\cup C_j\cup C_k)$.

\end{observation}

\noindent \textbf{Proof.}  By Observation \ref{cknotseparate} and \ref{bothstrongcj}, either $(C_i,C_j,C_k)$ is not separable or $C_i$ separates it.  If $(C_i,C_j,C_k)$ is not separable, then $\gamma_i,\gamma_j,\gamma_k$ are pairwise disjoint simple continuous curves that appear on the boundary of $conv(C_i\cup C_j\cup C_k)$ in clockwise order. Hence $(C_i,C_j)$ appears in consecutive clockwise order and $(C_i,C_k)$ appears in consecutive counterclockwise order on the boundary of $conv(C_i\cup C_j\cup C_k)$ and we are done.

If $C_i$ separates $(C_i,C_j,C_k)$, then $\gamma_j$ and $\gamma_k$ are simple continuous curves while $\gamma_i$ consists of two disjoint simple continuous curves $\gamma_{i_1}$ and $\gamma_{i_2}$.  Moreover $\gamma_{i_1},\gamma_j,\gamma_{i_2},\gamma_k$ are pairwise disjoint and appear in clockwise order.  Hence $(C_i,C_j)$ appears in consecutive clockwise order and $(C_i,C_k)$ appears in consecutive counterclockwise order on the boundary of $conv(C_i\cup C_j\cup C_k)$

$\hfill\square$

\noindent By a symmetric argument, we have the following.

\begin{observation}
\label{counterclockwisevec}
For $i < j< k$, if $(C_i,C_j,C_k)$ has only a strong-counterclockwise orientation, then $C_i\cup C_j\cup C_k$ lies to the left of $\vec{r}_{ij}$ and to the right of $\vec{l}_{ik}$.

\end{observation}

\subsection{The transitive property and its application}

\begin{lemma}\label{bothorientations}
Let $i < j< k< l$. If  the triples $(C_i,C_j,C_k)$ and $(C_j,C_k,C_l)$ have both strong orientations, then the same is true for $(C_i,C_j,C_l)$ and $(C_i,C_k,C_l)$.
\end{lemma}

\noindent{\bf Proof.}
By Observations \ref{bothstrongcj} and \ref{bothstrongvec},  both $C_i$ and $C_l$ must lie to the right of $\vec{l}_{jk}$ and to the left of $\vec{r}_{jk}$ as shown in Figure 2.

\begin{figure}[h]
\begin{center}
\includegraphics[width=180pt]{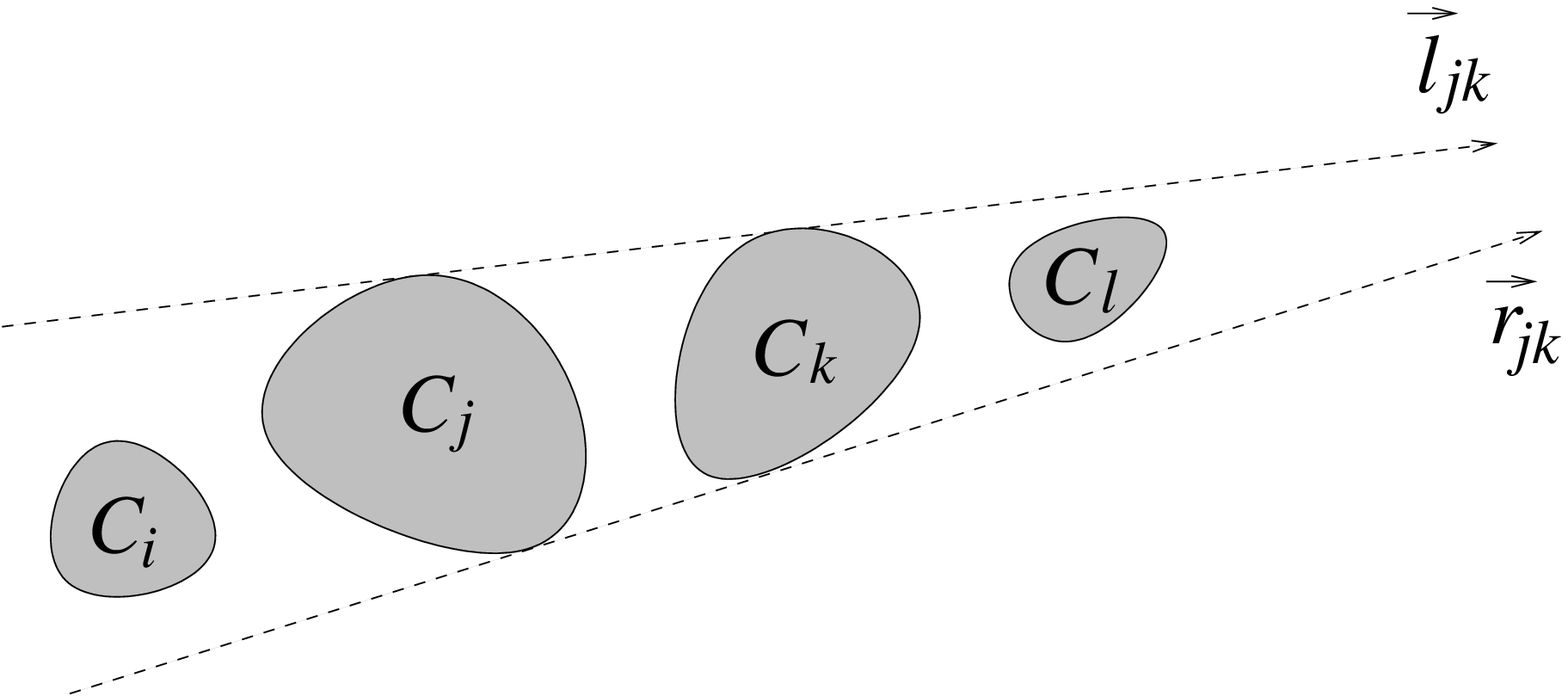}
  \caption{For Lemma \ref{bothorientations}.}
 \end{center}
\end{figure}

Hence $C_j$ separates the triple $(C_i,C_j,C_l)$ and $C_k$ separates the triple $(C_i,C_k,C_l)$, which implies $(C_i,C_j,C_l)$ and $(C_i,C_k,C_l)$ have both strong orientations.

$\hfill\square$

\begin{lemma}\label{oneorientation}
\label{case}
Let $i < j< k< l$. If the triples $(C_i,C_j,C_k)$ and $(C_j,C_k,C_l)$  have only strong-clockwise orientations (or only strong-counterclockwise orientations), then the same is true for $(C_i,C_j,C_l)$ and $(C_i,C_k,C_l)$.
\end{lemma}

\noindent{\bf Proof.}
By symmetry, it is sufficient to verify the statement in the case when the triples $(C_i,C_j,C_k)$ and $(C_j,C_k,C_l)$ have only strong-clockwise orientations.

\begin{claim}
\label{ikl}
$(C_i,C_k,C_l)$ has only a strong-clockwise orientation.
\end{claim}

\noindent Notice that we can make the following assumption.

\begin{center}
\begin{tabular}{rl}
($\ast$) & $C_l$ lies to the left of $\vec{r}_{ik}$ and to the right of $\vec{l}_{jk}$.
\end{tabular}
\end{center}

\noindent Indeed, by Observations \ref{bothstrongvec} and \ref{counterclockwisevec}, we can assume that $C_l$ lies to the left of $\vec{r}_{ik}$ since otherwise $(C_i,C_k,C_l)$ would not have a strong-counterclockwise orientation and we would be done, that is, $(C_i,C_k,C_l)$ has only a strong-clockwise orientation since every triple has at least one strong orientation.  Moreover by Observation \ref{clockwisevec}, $C_l$ must lie to the right of $\vec{l}_{jk}$ since $(C_j,C_k,C_l)$ has only a strong-clockwise orientation.

We define $L$ to be the vertical line through the left endpoint of $C_k$, $p_1$ to be the left endpoint of $C_k$, $p_2$ to be the first point on $\vec{l}_{jk}$ from $C_k$ in the direction of $\vec{l}_{jk}$, $p_3$ to be the first point on $\vec{r}_{ik}$ from $C_k$ in the direction of $\vec{r}_{ik}$, and $p_4$ to be the first point on $\vec{r}_{jk}$ from $C_k$ in the direction of $\vec{r}_{jk}$.  For simplicity, we will assume that $p_1,p_2,p_3,p_4$ are distinct, since otherwise it would just make the remaining analysis easier.

\begin{observation}
\label{p1234}
 $p_1,p_2,p_3,p_4$ appear in clockwise order along the boundary of $C_k$.

\end{observation}

\noindent \textbf{Proof.}   Let $\partial$ denote the boundary of $conv(C_j\cup C_k)$.  Since $C_j$ and $C_k$ are noncrossing, $C_k\cap \partial$ is the arc generated by moving along the boundary of $C_k$ from $p_2$ to $p_4$ in the clockwise direction.  Since the left endpoint of $C_j$ lies to the left of $L$ and since $p_1,p_2,p_4$ are all distinct, $p_1$ must lie in the interior of $conv(C_j\cup C_k)$.  Hence $p_1,p_2,p_4$ appear in clockwise order along the boundary of $C_k$.  Therefore it suffices to show that points $p_2,p_3,p_4$ appear in clockwise order.  By Observation \ref{clockwisevec}, $C_j$ must lie to the left of $\vec{r}_{ik}$ and therefore $p_3 \in C_k\cap\partial$.  Since $p_2,p_3,p_4$ are distinct, $p_2,p_3,p_4$ appear in clockwise order along the boundary of $C_k$ and this completes the proof.

$\hfill\square$

By the ordering on $\mathcal{C}$ and by $(\ast)$, $C_l$ must lie to the right of $L$, to the right of $\vec{l}_{jk}$, and to the left of $\vec{r}_{ik}$.  Hence $C_l\setminus C_k$ must lie in one of the three regions defined as follows.  Let \emph{Region I} be the region enclosed by $L$, $\vec{l}_{jk}$, and the arc generated by moving from $p_1$ to $p_2$ along the boundary of $C_k$ in the clockwise direction, whose interior is disjoint from $C_k$.  Let \emph{Region II} be the region enclosed by $\vec{l}_{jk}$, $\vec{r}_{ik}$, and the arc generated by moving from $p_2$ to $p_3$ along the boundary of $C_k$ in the clockwise direction, whose interior is disjoint from $C_k$.  Let \emph{Region III} be the region enclosed by $L$, $\vec{r}_{ik}$, and the arc generated by moving from $p_3$ to $p_1$ along the boundary of $C_k$ in the clockwise direction, whose interior is disjoint from $C_k$.  See Figure \ref{analysis}.

\begin{observation}
\label{region1}
$C_l\setminus C_k$ does not lie in Region I.

\end{observation}

\noindent\textbf{Proof.} For sake of contradiction, suppose $C_l\setminus C_k$ lies in Region I.  Then the proof falls into two cases.

\medskip

\noindent \emph{Case 1.}  Suppose that $C_j$ does not intersect Region I.  This implies that the tangent point of $C_j$ on $\vec{l}_{jk}$ lies to the left of $L$. Therefore, the triangle $T_1$ whose vertices is this  tangent point, $p_1$, and $p_2$ contains Region I. Since the vertices of $T_1$ are in $C_j \cup C_k$, the triangle $T_1$ is a subset of the convex hull $conv(C_j \cup C_k)$. Hence,
$$C_l \subset (\textnormal{Region I}\cup C_k) \subset (T_1 \cup C_k) \subset conv(C_j\cup C_k),$$
\noindent which is a contradiction.

\medskip

\noindent \emph{Case 2.}  Suppose $C_j$ does intersect Region I.  Let $h$ be the line that goes through the left endpoint of $C_j$ and a point on $C_j\cap \vec{l}_{jk}$.  If $C_l\setminus C_j$ lies below $h$, then $C_l\subset conv(C_j\cup C_k)$ which is a contradiction.  If $C_l\setminus C_k$ lies above $h$, then $C_j$ separates $(C_j,C_k,C_l)$ and by Observation \ref{ciseparates}, $(C_j,C_k,C_l)$ has a strong counter-clockwise orientation which is also a contradiction.  See Figure \ref{R1v2}.

 \begin{figure}[h]
  \centering
  \subfigure[$C_l\setminus C_k$ must lie in either Region I, II, or III.]{\label{analysis}\includegraphics[width=0.26\textwidth]{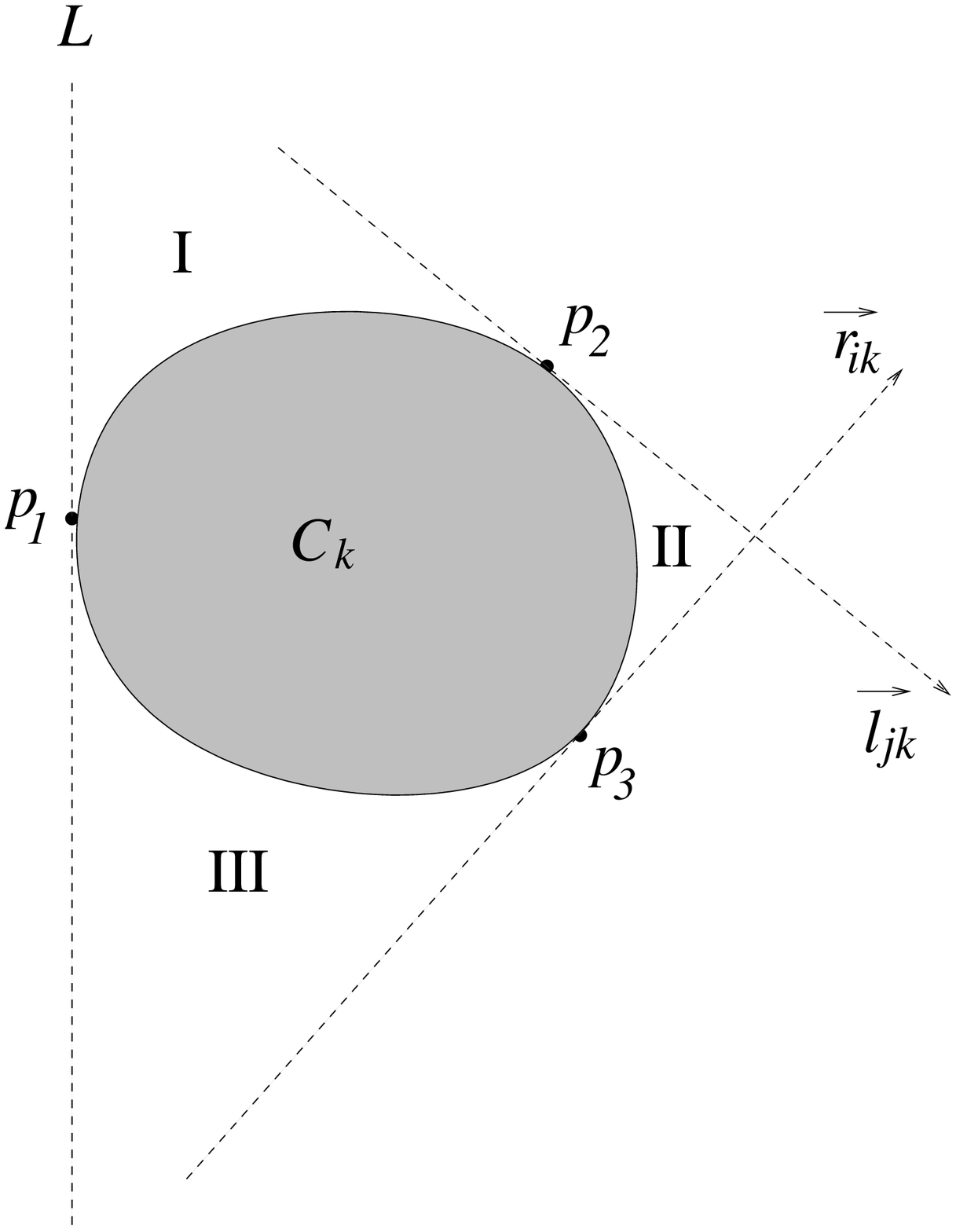}}\hspace{1cm}
                        \subfigure[$(C_j,C_k,C_l)$ does not have a strong-clockwise orientation.]{\label{R1v2}\includegraphics[width=0.30\textwidth]{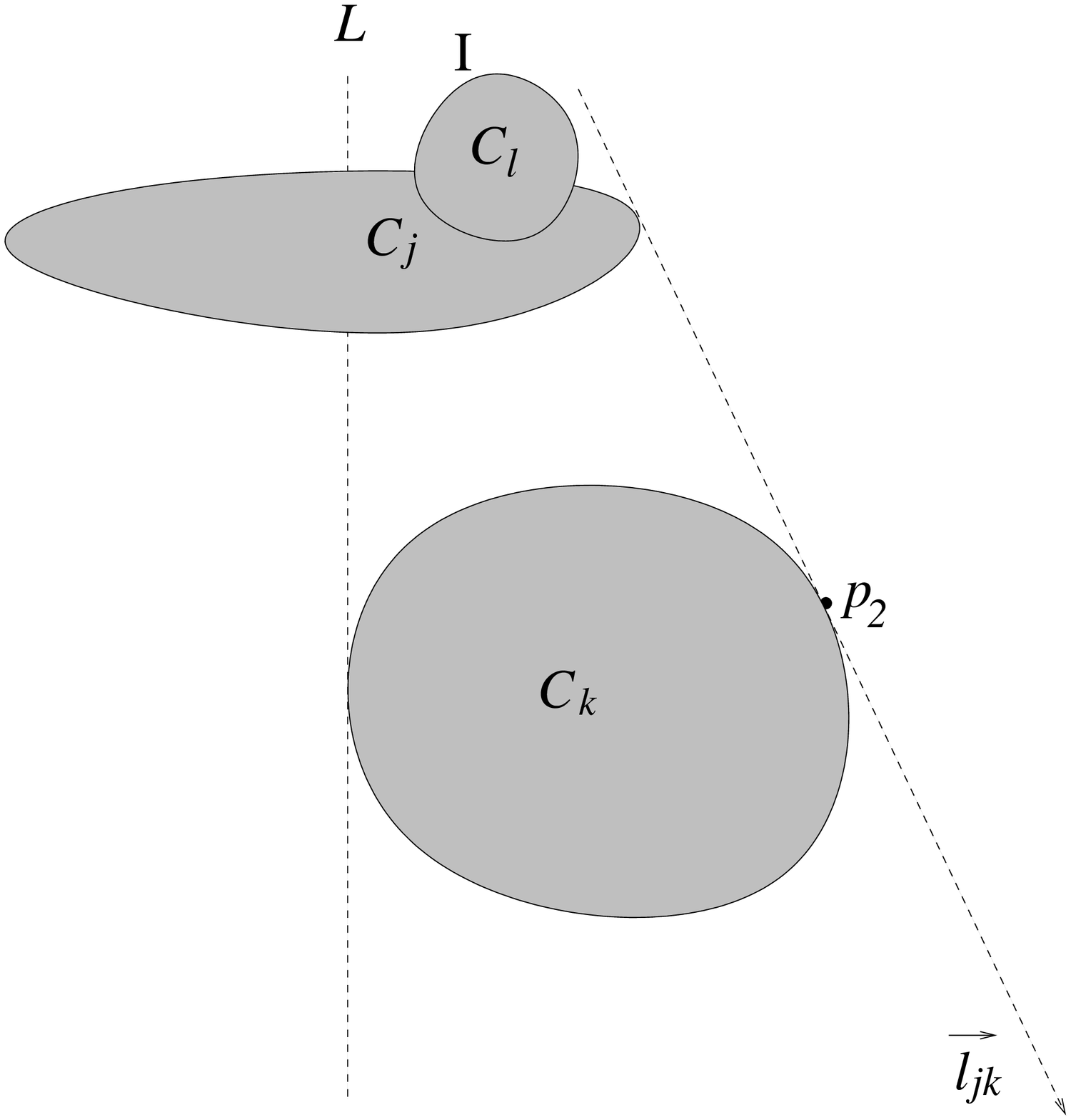}}
  \caption{Regions I, II, and III.}
  \label{fR1}
\end{figure}

$\hfill\square$

\begin{observation}

$C_l\setminus C_k$ does not lie in Region II.

  \end{observation}

\noindent \textbf{Proof.}  For sake of contradiction, suppose $C_l\setminus C_k$ lies in Region II.  By Observation \ref{p1234}, $p_2,p_3,p_4$ appear in clockwise order.  Hence $C_l$ lies to the left of $\vec{r}_{jk}$.  This implies that $C_k$ separates the triple $(C_j,C_k,C_l)$.  By Observation \ref{bothstrongcj}, $(C_j,C_k,C_l)$ has both strong orientations which is a contradiction.  See Figure \ref{R2}.

$\hfill\square$

Therefore $C_l\setminus C_k$ must lie in Region III.  Suppose for contradiction that $C_i$ does not intersect Region III. Then the tangent point of $C_i$ on $\vec{r}_{ik}$ lies to the left of $L$. Therefore, the triangle $T_2$ whose vertices is this tangent point, $p_1$, and $p_3$ contains Region III. Since the vertices of $T_2$ are in $C_i \cup C_k$, the triangle $T_2$ is a subset of the convex hull $conv(C_i \cup C_k)$. Hence,
$$C_l \subset (\textnormal{Region III}\cup C_k) \subset (T_2 \cup C_k) \subset conv(C_i\cup C_k),$$
\noindent which is a contradiction. Hence, $C_i$ must intersect Region III. Moreover, $C_l\setminus C_i$ must lie in the lower half-plane generated by the line that goes through the left endpoint of $C_i$ and a point on $C_i\cap \vec{r}_{ik}$ as otherwise again $C_{l} \subset conv(C_i \cup C_k)$. Since $C_i$ separates $(C_i,C_k,C_l)$ and by Observation \ref{ciseparates}, $(C_i,C_k,C_l)$ has only a strong-clockwise orientation.  See Figure \ref{strongclock2}.  This completes the proof of Claim \ref{ikl}.

 \begin{figure}[h]
  \centering
 \subfigure[$C_l\setminus C_k$ lies in Region II implies $(C_j,C_k,C_l)$ has both strong orientations.]{\label{R2}\includegraphics[width=0.33\textwidth]{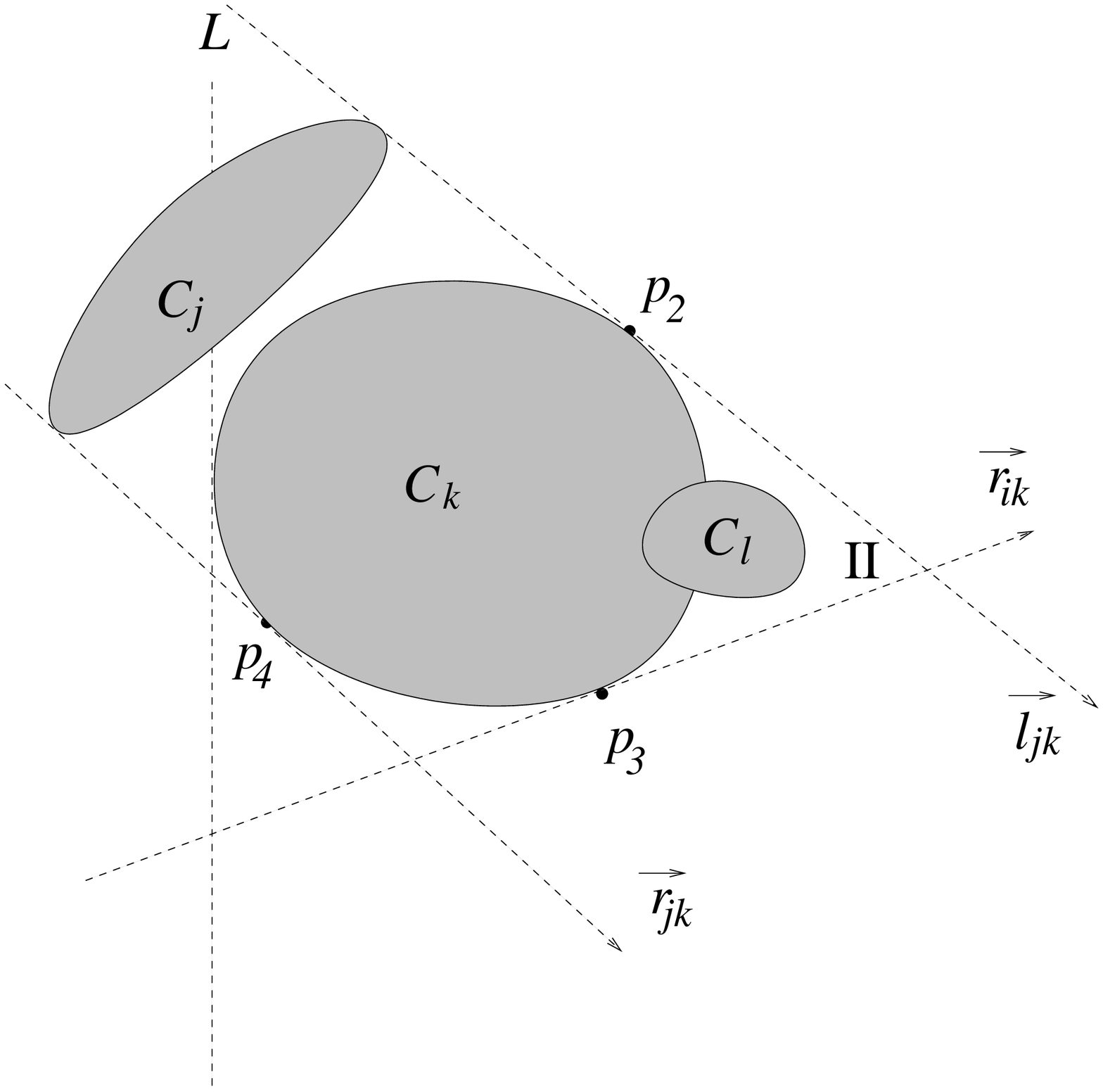}}\hspace{1cm}
                        \subfigure[$(C_i,C_k,C_l)$ has only a strong-clockwise orientation.]{\label{strongclock2}\includegraphics[width=0.28\textwidth]{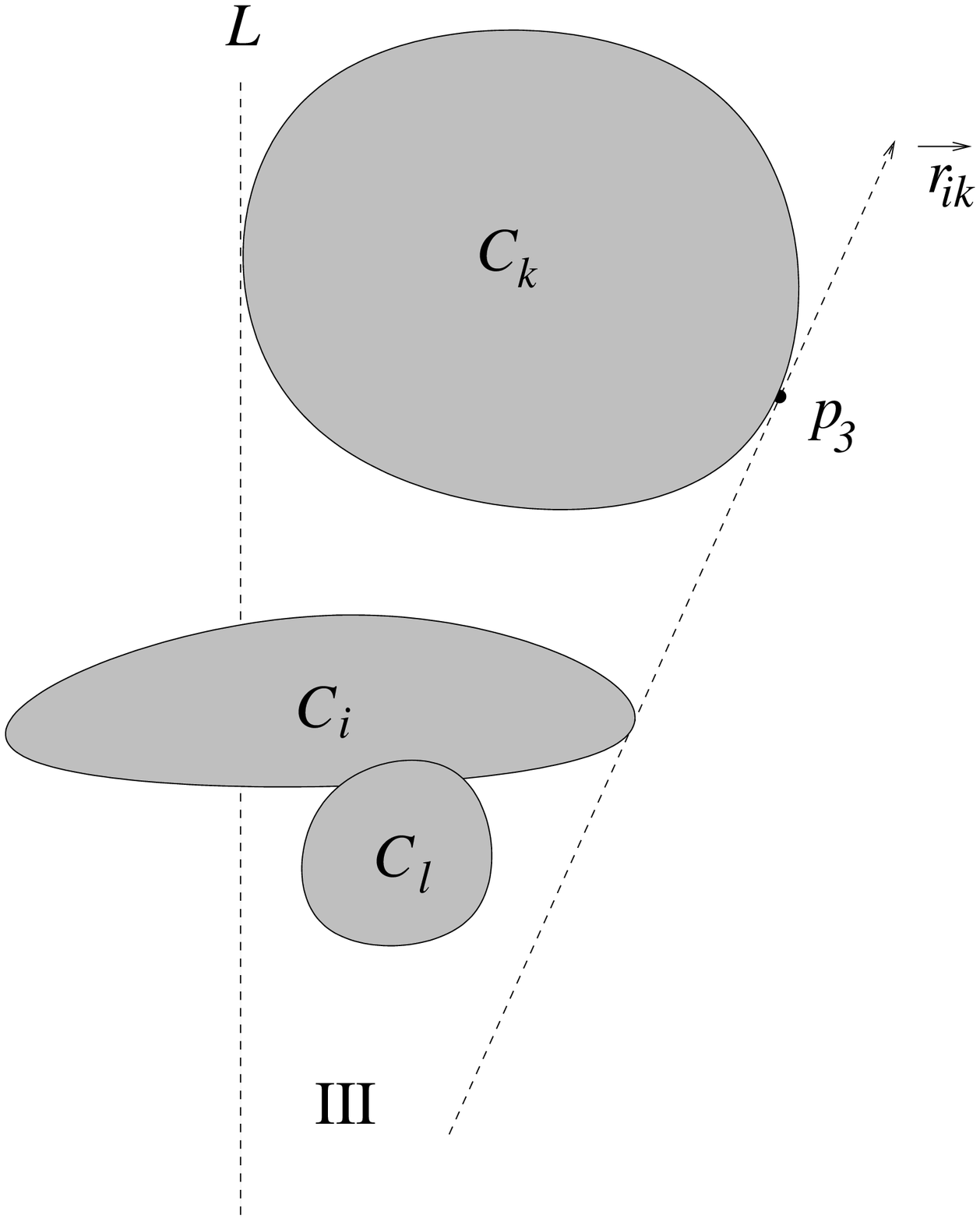}}
                        \caption{ Regions II and III.}
  \label{strongclock12}
\end{figure}

 $\hfill\square$

Now we will show that $(C_i,C_j,C_l)$ has only a strong-clockwise orientation by a very similar argument.

\begin{claim}
\label{ijl}
$(C_i,C_j,C_l)$ has only a strong-clockwise orientation.

 \end{claim}

\noindent Just as before, we can assume the following.

 \begin{center}
\begin{tabular}{rl}
($\ast\ast$) & $C_l$ lies to the left of $\vec{r}_{ij}$ and to the right of $\vec{l}_{jk}$.
\end{tabular}
\end{center}

 \noindent Indeed, by Observation \ref{counterclockwisevec}, we can assume that $C_l$ lies to the left of $\vec{r}_{ij}$ since otherwise $(C_i,C_j,C_l)$ would not have a strong-counterclockwise orientation and we would be done, that is, $(C_i,C_j,C_l)$ has only a strong-clockwise orientation since every triple has at least one strong orientation.  Moreover by Observation \ref{clockwisevec}, $C_l$ must lie to the right of $\vec{l}_{jk}$ since $(C_j,C_k,C_l)$ has only a strong-clockwise orientation.

Since $(C_i,C_j,C_k)$ has only a strong-clockwise orientation, then either $(C_i,C_j,C_k)$ is not separable, or $C_i$ separates it.  Now the proof falls into two cases.

\medskip

\noindent{\sc Case 1:}  Assume that $(C_i,C_j,C_k)$ is not separable.    Then we define $L'$ to be the vertical line through the left endpoint of $C_j$, $p_5$ to be the left endpoint of $C_j$, $p_6$ to be the first point on $\vec{l}_{jk}$ from $C_j$ in the direction of $\vec{l}_{jk}$, $p_7$ to be the first point on $\vec{r}_{ij}$ from $C_j$ in the direction of $\vec{r}_{ij}$, and $p_8$ to be the first point on $\vec{r}_{jk}$ from $C_j$ in the direction of $\vec{r}_{jk}$.  For simplicity, we will assume that $p_5,p_6,p_7,p_8$ are distinct, since otherwise it would just make the remaining analysis easier.  Now we make the following observation.

\begin{observation}
\label{p5678v1}
$p_5,p_6,p_7,p_8$ appear in clockwise order along the boundary of $C_j$.

\end{observation}

\noindent \textbf{Proof.}  Let $\partial$ denote the boundary of $conv(C_j\cup C_k)$.  Then $C_j\cap \partial$ is the arc generated by moving along the boundary of $C_j$ from $p_8$ to $p_6$ in the clockwise direction.  Since $p_5$ is the leftmost point of $conv(C_j\cup C_k)$, $p_5 \in C_j\cap  \partial$.  Hence $p_5,p_6,p_8$ appear in clockwise order along the boundary of $C_j$.  Therefore it suffices to show that $p_6,p_7,p_8$ appear in clockwise order.  Since $(C_i,C_j,C_k)$ is not separable and has a strong-clockwise orientation, $C_k$ does not lie completely to the left of $\vec{r}_{ij}$.  Since $p_6,p_7,p_8$ are all distinct, $p_7$ lies in the interior of $conv(C_j\cup C_k)$.  Thus $p_6,p_7,p_8$ appear in clockwise order along the boundary of $C_j$ and this completes the proof.

$\hfill\square$

By $(\ast\ast)$ and the ordering on $\mathcal{C}$, $C_l$ must lie to the right of $L'$, to the left of $\vec{r}_{ij}$, and to the right of $\vec{l}_{jk}$.  Therefore $C_l\setminus C_j$ must lie in one of the three regions defined as follows.  Let \emph{Region IV} be the region enclosed by $\vec{l}_{jk}$, $L'$, and the arc generated by moving from $p_5$ to $p_6$ along the boundary of $C_j$ in the clockwise direction, whose interior is disjoint from $C_j$.  Let \emph{Region V} be the region enclosed by $\vec{l}_{jk}$, $\vec{r}_{ij}$, and the arc generated by moving from $p_6$ to $p_7$ along the boundary of $C_j$ in the clockwise direction, whose interior is disjoint from $C_j$.  Let \emph{Region VI} be the region enclosed by $L'$, $\vec{r}_{ij}$, and the arc generated by moving from $p_7$ to $p_5$ along the boundary of $C_j$ in the clockwise direction, whose interior is disjoint from $C_j$.  See Figures \ref{R4t} and \ref{R5t}.

 \begin{figure}[h]
  \centering
\subfigure[$(C_j,C_k,C_l)$ separable.]{\label{R4t}\includegraphics[width=0.25\textwidth]{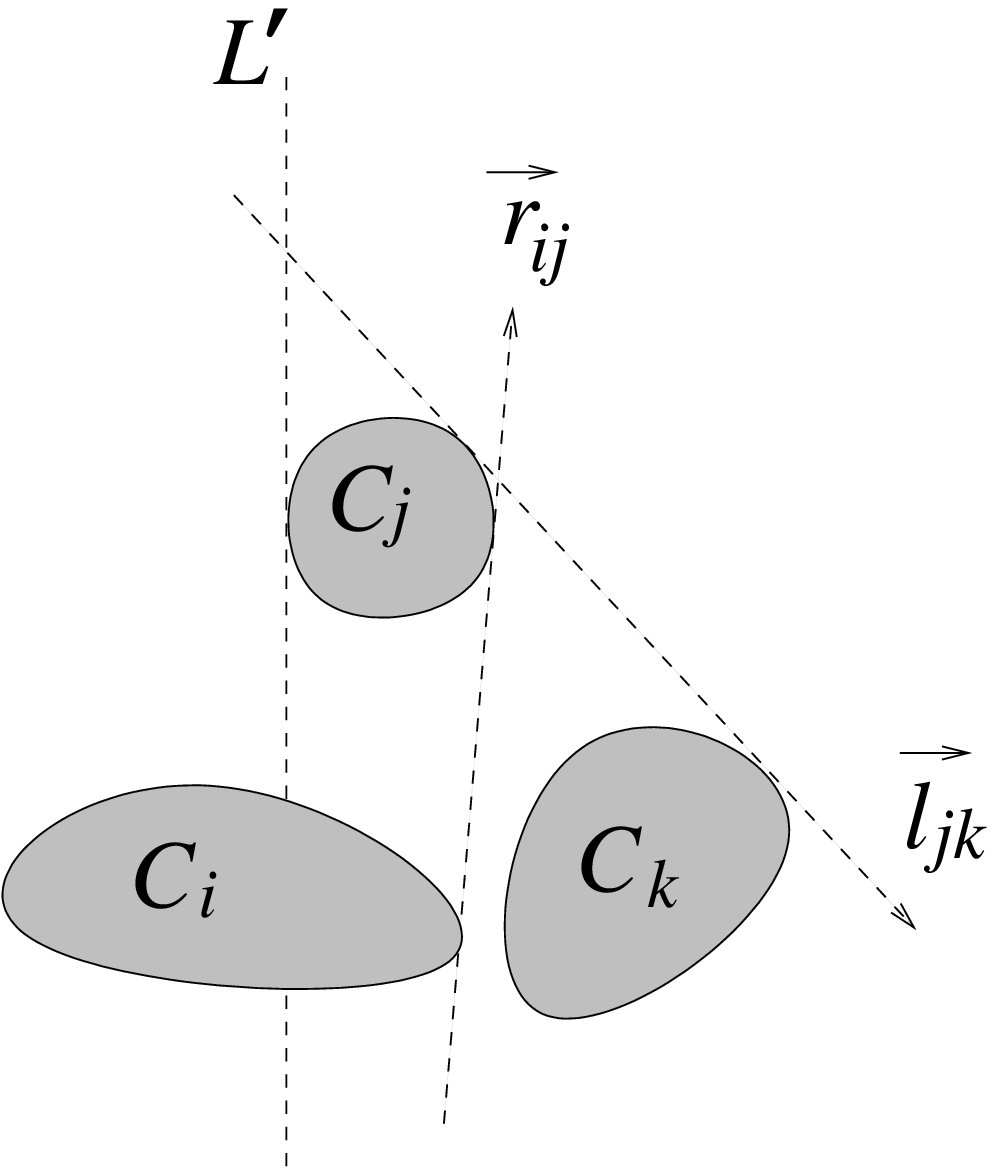}}\hspace{1cm}
                        \subfigure[Regions IV, V, and VI.]{\label{R5t}\includegraphics[width=0.3\textwidth]{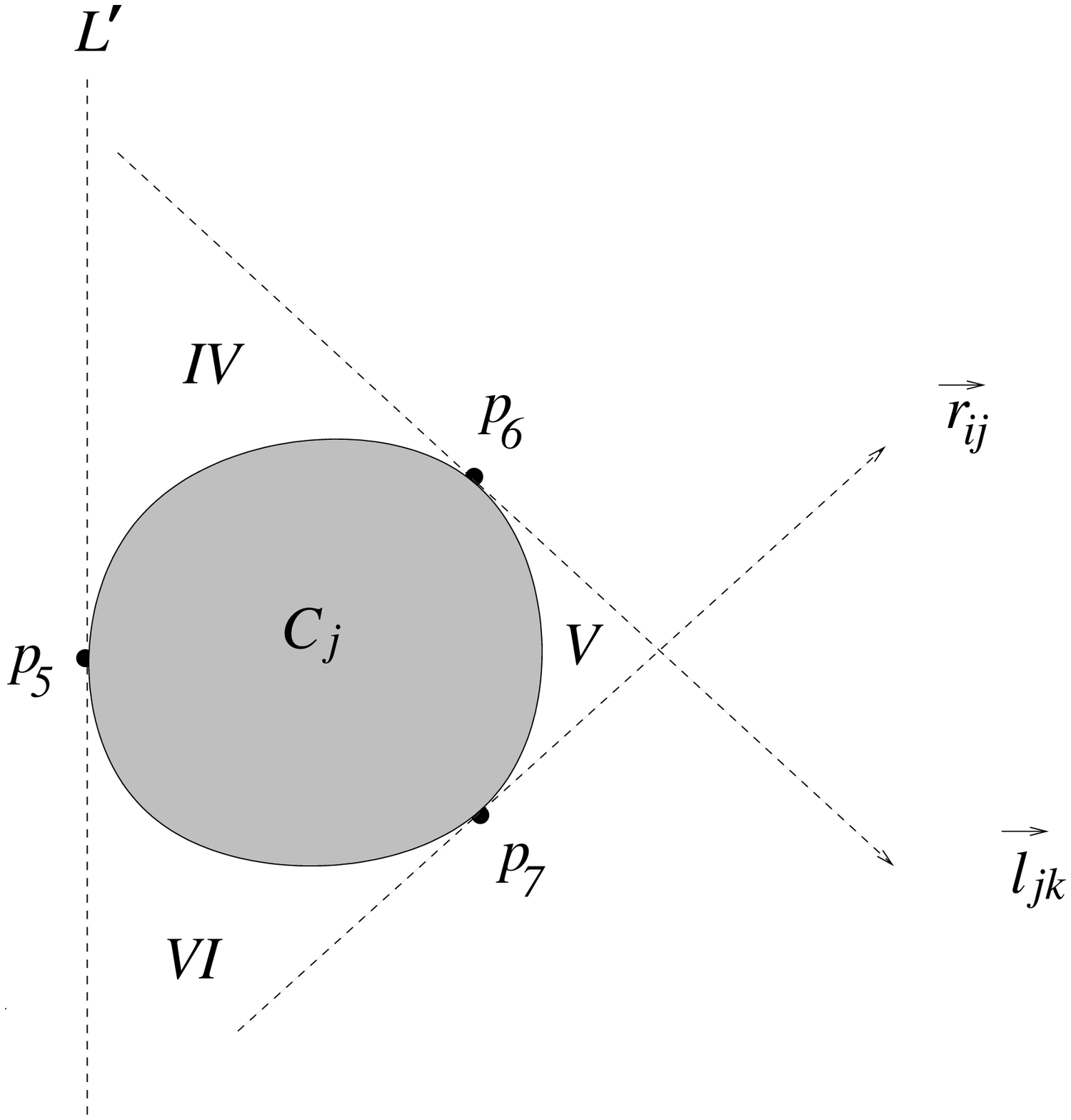}}
                        \caption{Regions IV, V, and VI.}
  \label{ijk12}
\end{figure}

\begin{observation}

$C_l\setminus C_j$ cannot lie in Region IV

 \end{observation}

\noindent \textbf{Proof.} For sake of contradiction, assume that $C_l\setminus C_j$ does lie in Region IV.  Since $p_5,p_6,p_8$ appear in clockwise order along the boundary of $C_j$, $C_j$ separates the triple $(C_j,C_k,C_l)$.  Moreover, $C_l\setminus C_j$ lies in the upper half-plane generated by the line that goes through $p_5$ and $p_6$.  By Observation \ref{ciseparates}, $(C_j,C_k,C_l)$ has only a strong-counterclockwise orientation and we have a contradiction.  See Figure \ref{R4}.

 $\hfill\square$

 \begin{observation}

 $C_l\setminus C_j$ cannot lie in Region V.

  \end{observation}

\noindent \textbf{Proof.}  For sake of contradiction, assume that $C_l\setminus C_j$ does lie in Region V.  Since $C_l$ lies to the left of $\vec{r}_{ij}$ and since $p_7, p_8$ appear in clockwise order along the boundary of $C_j$, we can conclude that $C_l$ lies to the left of $\vec{r}_{jk}$.  Since $C_l\not\subset conv(C_j\cup C_k)$, $C_k$ separates the triple $(C_j,C_k,C_l)$.  By Observation \ref{bothstrongcj}, $(C_j,C_k,C_l)$ has both strong orientations and we have a contradiction.  See Figure \ref{R5}.

   $\hfill\square$

  Therefore $C_l\setminus C_j$ must lie in Region VI.  Suppose for contradiction that $C_i$ does not intersect Region VI. Then the tangent point of $C_i$ on $\vec{r}_{ij}$ lies to the left of $L'$. Therefore, the triangle $T_3$ whose vertices is this tangent point, $p_5$, and $p_7$ contains Region VI. Since the vertices of $T_3$ are in $C_i \cup C_j$, the triangle $T_3$ is a subset of the convex hull $conv(C_i \cup C_j)$. Hence,
$$C_l \subset (\textnormal{Region VI}\cup C_j) \subset (T_3 \cup C_j) \subset conv(C_i\cup C_j),$$
\noindent which is a contradiction. Therefore $C_i$ must intersect Region VI and $C_l\setminus C_i$ must lie in the lower half-plane generated by the line that goes through the left endpoint of $C_i$ and a point on $C_i\cap \vec{r}_{ij}$.  Therefore $C_i$ separates $(C_i,C_j,C_l)$ and by Observation \ref{ciseparates}, $(C_i,C_j,C_l)$ has only a strong-clockwise orientation.  See Figure \ref{R6}.

 \begin{figure}[h]
  \centering
\subfigure[In Region IV, $(C_j,C_k,C_l)$ has a strong-counterclockwise orientation.]{\label{R4}\includegraphics[width=0.28\textwidth]{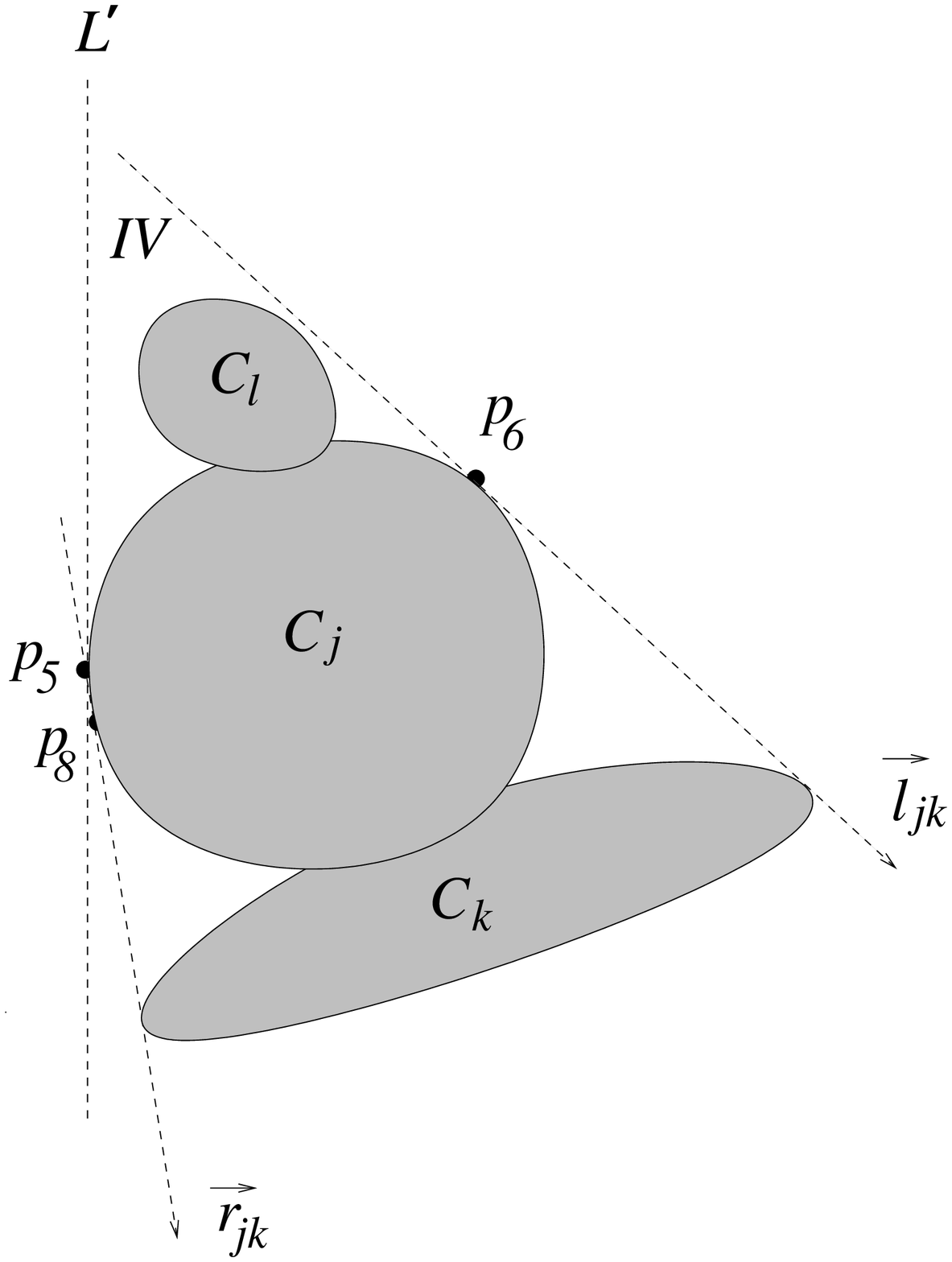}}\hspace{1cm}
                        \subfigure[In Region V, $C_l \subset conv(C_j\cup C_k)$.]{\label{R5}\includegraphics[width=0.34\textwidth]{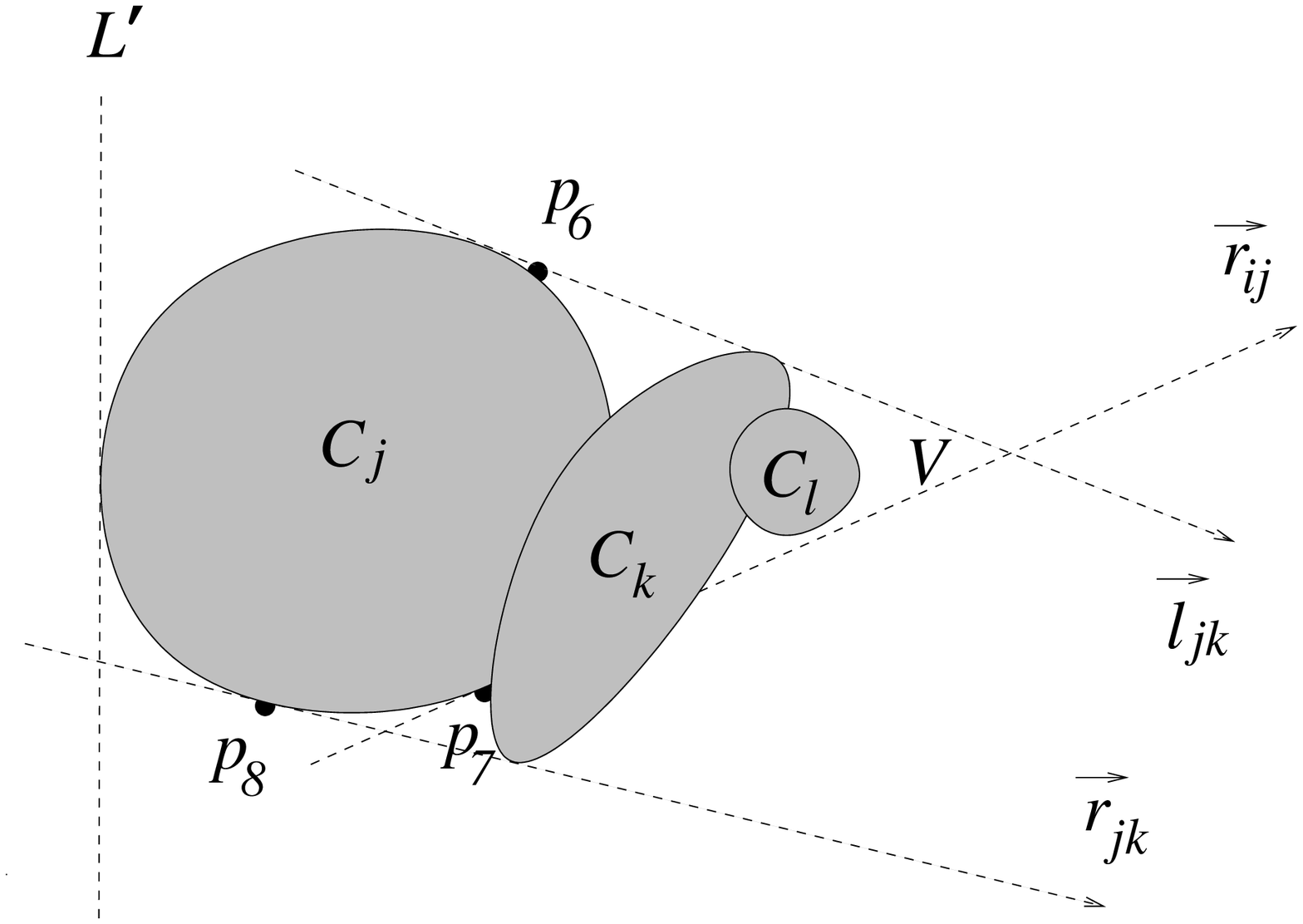}}\hspace{1cm}
                         \subfigure[In Region VI, $(C_i,C_j,C_l)$ has only a strong-clockwise orientation.]{\label{R6}\includegraphics[width=0.2\textwidth]{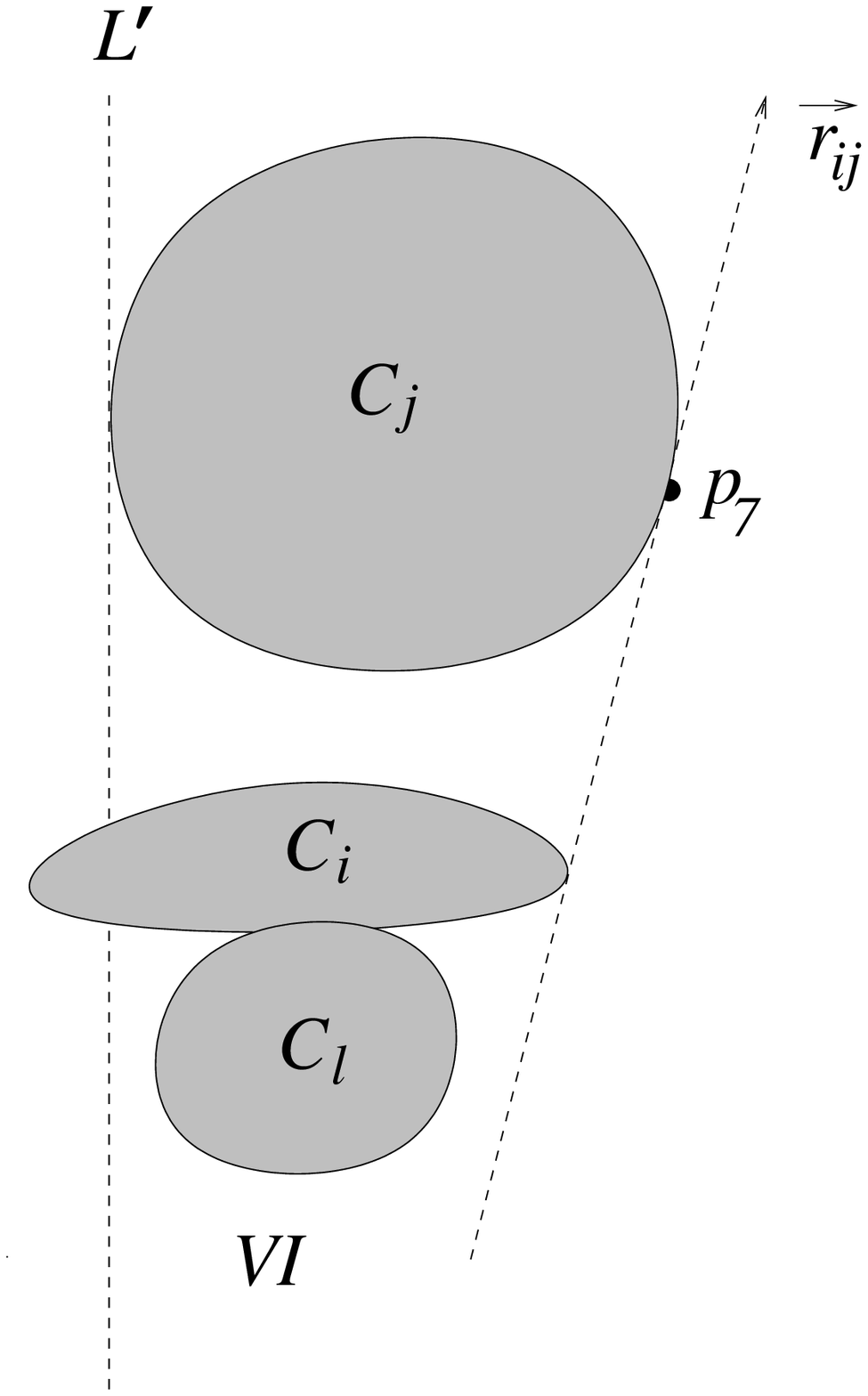}}
                        \caption{Regions IV, V, and VI.}
  \label{ijk12}
\end{figure}

\medskip

\noindent{\sc Case 2:}  Assume that $C_i$ separates $(C_i,C_j,C_k)$.

 \begin{observation}
 \label{p5678v2}
 $p_5,p_7,p_6$ appear in clockwise order along the boundary of $C_j$.

 \end{observation}

 \noindent \textbf{Proof.}  Let $\beta$ denote the boundary of $conv(C_i\cup C_j)$.  Since $C_i$ separates $(C_i,C_j,C_k)$, the arc generated by moving along the boundary of $C_j$ from $p_5$ to $p_6$ in the clockwise direction must contain $C_j\cap \beta$.  Since $p_7 \in C_j\cap \beta$, $p_5,p_7,p_6$ appear in clockwise order along the boundary of $C_j$.

 $\hfill\square$

By $(\ast\ast)$ and the ordering on $\mathcal{C}$, $C_l$ must lie to the right of $L'$ and to the right of $\vec{l}_{jk}$.  Hence $C_l\setminus C_j$ must lie in one of the following two regions.  Let \emph{Region VII} be the region enclosed by $\vec{l}_{jk}$, $L'$, and the arc generated by moving from $p_5$ to $p_6$ along the boundary of $C_j$ in the clockwise direction, whose interior is disjoint from $C_j$.  Finally let \emph{Region VIII} be the region enclosed by $L'$, $\vec{l}_{jk}$, and the arc generated by moving from $p_6$ to $p_5$ along the boundary of $C_j$ in the clockwise direction, whose interior is disjoint from $C_j$.  See Figures \ref{ijk2} and \ref{R789}.

 \begin{figure}[h]
  \centering
\subfigure[$C_i$ separates $(C_j,C_k,C_l)$.]{\label{ijk2}\includegraphics[width=0.2\textwidth]{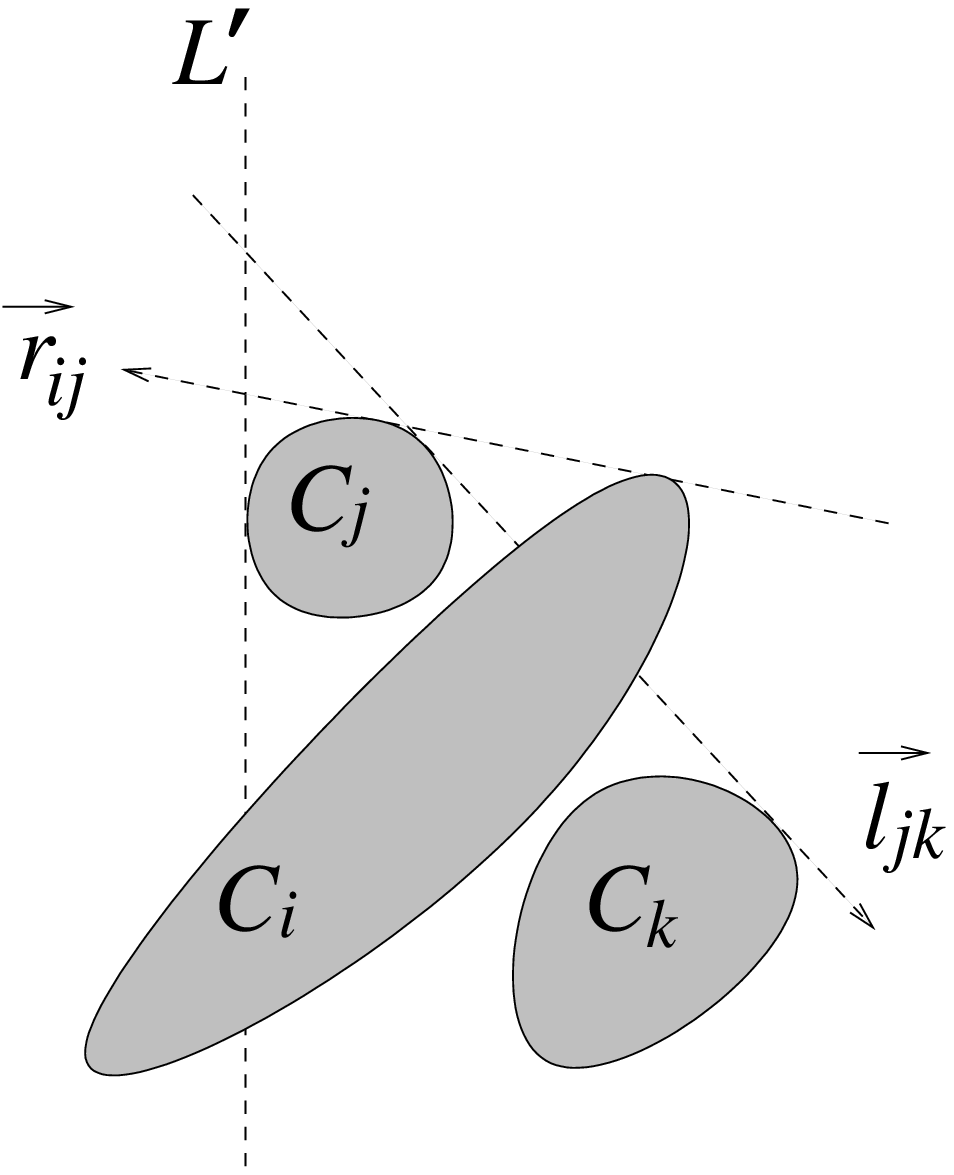}}\hspace{2cm}
                        \subfigure[Regions VII and VIII.]{\label{R789}\includegraphics[width=0.25\textwidth]{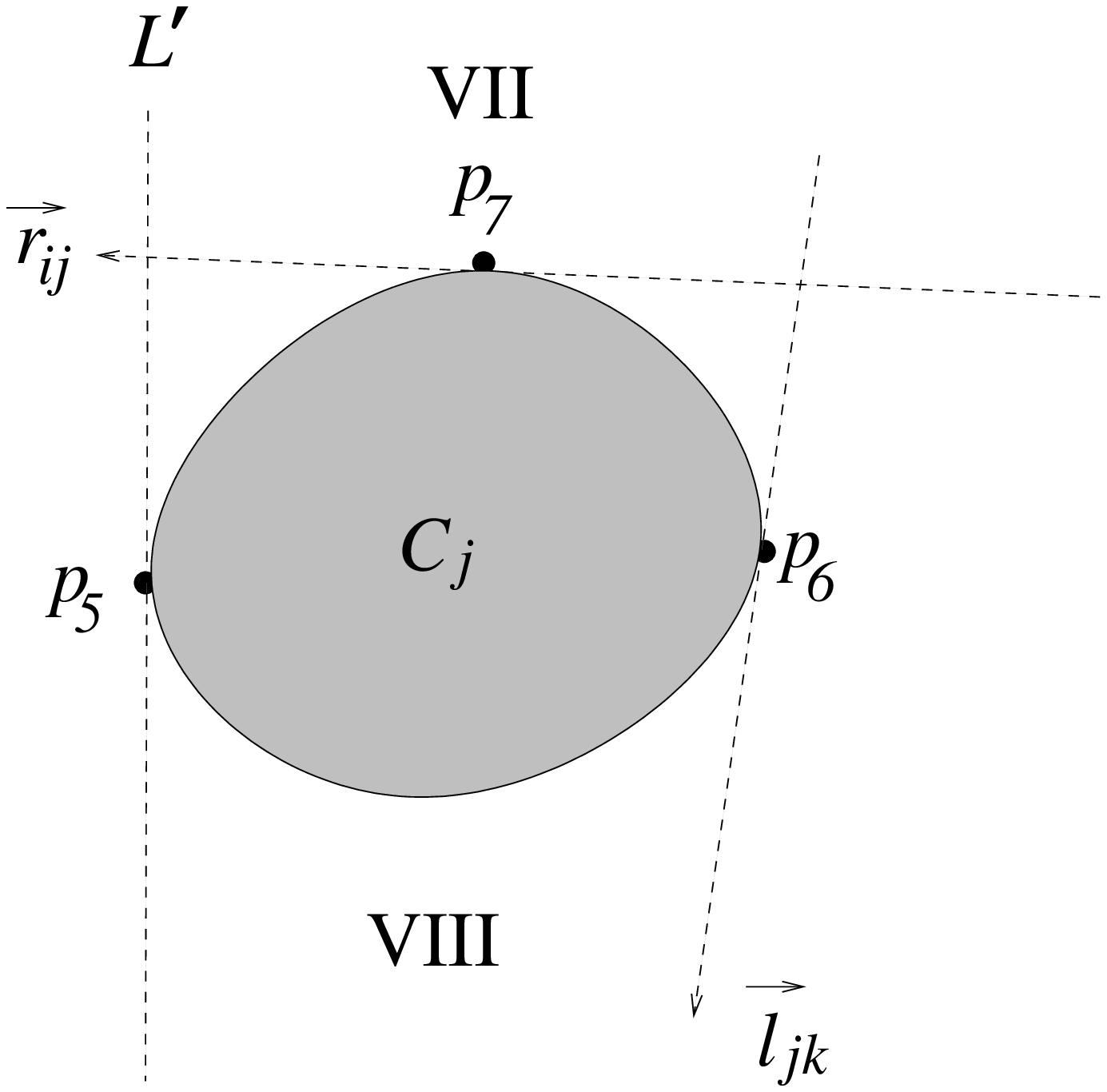}}
                        \caption{$C_i$ separates $(C_i,C_j,C_k)$.}
  \label{fR789}
\end{figure}

\begin{observation}
$C_l\setminus C_j$ cannot lie in Regions VII.
\end{observation}

\noindent \textbf{Proof.}  For sake of contradiction, suppose $C_l\setminus C_j$ lies in Region VII.  Then $C_j$ separates the triple $(C_j,C_k,C_l)$, and $C_l\setminus C_j$ lies in the upper half-plane generated by the line that goes through the points $p_5$ and $p_6$.  By Observation \ref{ciseparates}, ($C_j,C_k,C_l)$ has only a strong-counterclockwise orientation and we have a contradiction.  See Figure \ref{R8}.

$\hfill\square$

 Therefore $C_l\setminus C_j$ must lie in Region VIII.  Since $C_i$ separates the triple $(C_i,C_j,C_k)$, which has only a strong-clockwise orientation, $C_i$ must intersect Region VIII and $\vec{l}_{jk}$.  Since $C_l\not\subset conv(C_i\cup C_j)$, $C_l\setminus C_i$ must lie in the lower half plane generated by the line that goes through the left endpoint of $C_i$ and a point on $C_i\cap \vec{r}_{ij}$.  Thus $C_i$ separates $(C_i,C_j,C_l)$ and by Observation \ref{ciseparates}, $(C_i,C_j,C_l)$ has only a strong-clockwise orientation.  See Figure \ref{R9}.  This completes the proof of Claim \ref{ijl} and Lemma \ref{oneorientation}.

 \begin{figure}[h]
  \centering
                        \subfigure[In Region VII, $(C_j,C_k,C_l)$ has a strong-counterclockwise orientation.]{\label{R8}\includegraphics[width=0.25\textwidth]{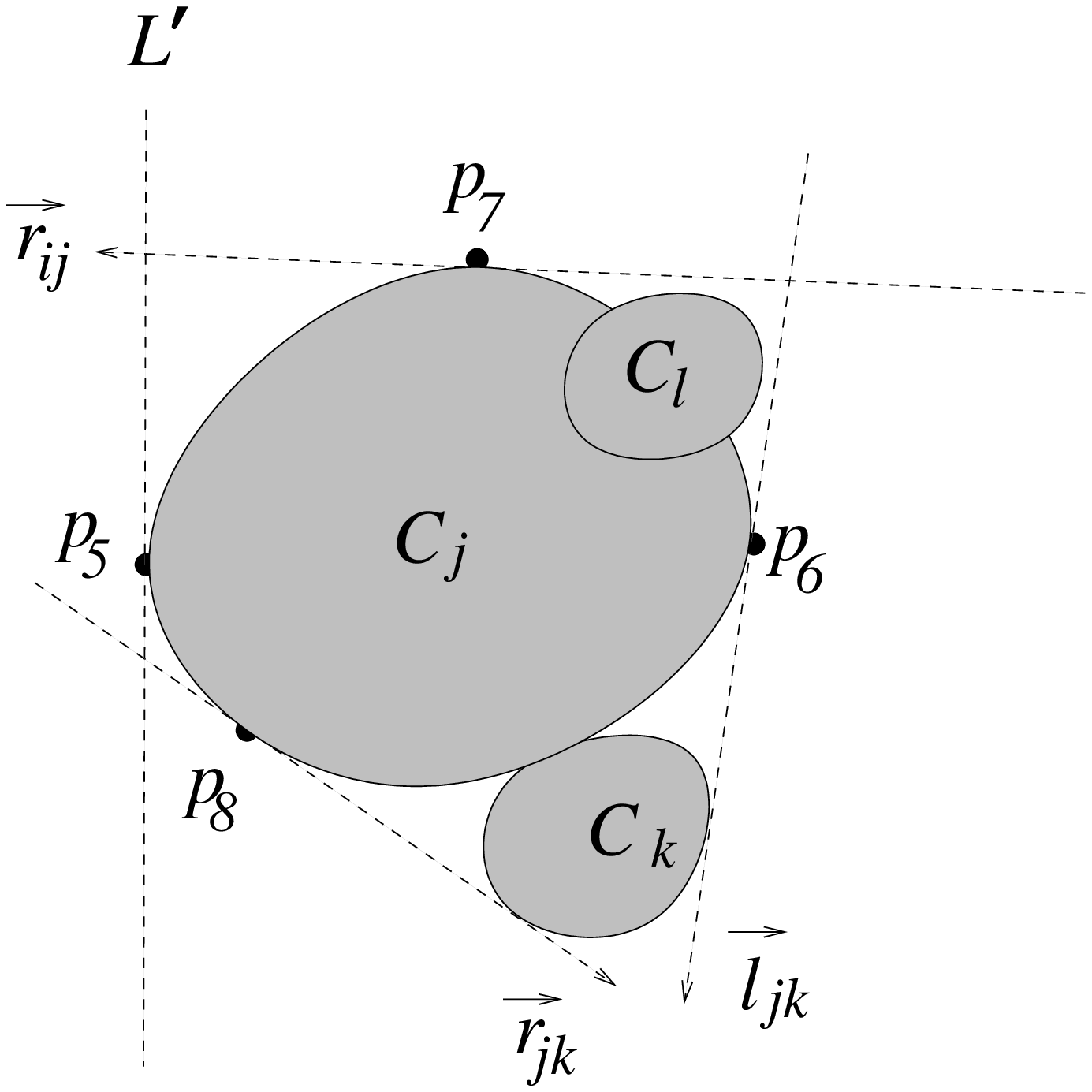}}\hspace{1cm}
                         \subfigure[In Region VIII, $(C_i,C_j,C_l)$ has only a strong-clockwise orientation.]{\label{R9}\includegraphics[width=0.3\textwidth]{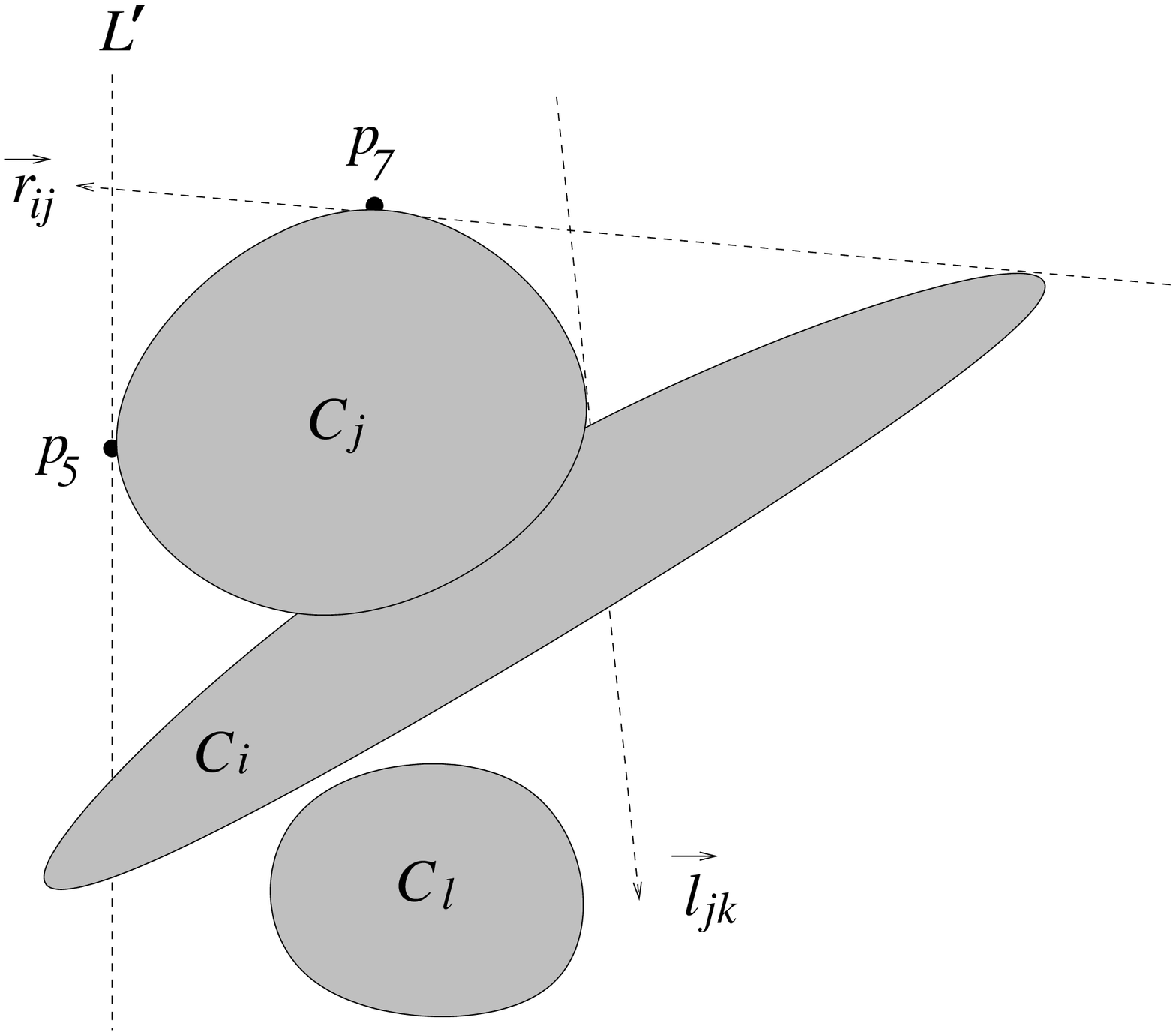}}
                        \caption{Regions VII and VIII.}
  \label{final}
\end{figure}

 $\hfill\square$

\medskip

Now we are ready to define a 3-coloring of the complete $3$-uniform ordered hypergraph $K^3_{\mathcal{C}}$ formed by all triples in $\mathcal{C}=\{C_1,\ldots,C_N\}$. A triple $(C_i,C_j,C_k)$, $i<j<k$ is colored
\begin{enumerate}

\item  {\em red} if it has only a strong-clockwise orientation;

\item  {\em blue} if it has only a strong-counterclockwise orientation;

\item  {\em green} if it has both strong orientations.
\end{enumerate}
Lemmas \ref{bothorientations} and \ref{oneorientation} imply that the above coloring is transitive.

\medskip

Since $N= N_3(3,n)$, the edges of $K^3_{\mathcal{C}}$, i.e., the triples of $\mathcal{C}$, determine a monochromatic monotone path of length $n$. As this coloring is transitive,  Lemma~\ref{complete} implies that the members of this path induce a monochromatic {\em complete} $3$-uniform subhypergraph $K\subseteq K^3_{\mathcal{C}}$. It follows from the definition of the coloring that all hyperedges of $K$ are triples which have the same strong orientations. The proof of Theorem~\ref{convex} is now a direct consequence of the following statement from \cite{alfredo} and Theorem \ref{main1}, which gives an upper bound on $N=N_3(3,n)$.

\begin{lemma}
\label{order}
If a family of noncrossing convex bodies in general position in the plane can be ordered in such a way that every triple has a clockwise orientation (or every triple has a counterclockwise orientation), then all members of the family are in convex position.
\end{lemma}
$\hfill\square$

\section{Concluding Remarks} \label{concluding}

\noindent $1$. We proved estimates on the minimum number of vertices $N_k(q,n)$ of a complete ordered $k$-uniform hypergraph for which every $q$-coloring of its edges has a monochromatic monotone path of length $n$. In particular, for $q \geq 3$ and $n \geq q+2$, we have $$2^{(n/q)^{q-1}} \leq N_3(q,n) \leq n^{n^{q-1}}.$$ It would be interesting to close the gap between the lower and upper bound.

\vspace{0.1cm}
\noindent $2$. We proved that the smallest integer $M(n)$ such that every family of $M(n)$ noncrossing convex bodies in general position in the plane has $n$ members in convex position satisfies $M(n) \leq N_3(3,n) \leq n^{n^2}$. It would be interesting to improve the bound on $M(n)$ further and, in particular, decide whether $M(n)$ grow exponential in $n$.

\vspace{0.1cm}
\noindent $3$. We proved that triples of noncrossing convex bodies with \emph{only} a strong-clockwise orientation have the transitive property. However it is easy to see that triples with a  strong-clockwise orientation do not.  In particular, it is possible to have four convex bodies ordered from left to right according to their left endpoints, such that $(C_1,C_2,C_3)$ and $(C_2,C_3,C_4)$ have a strong-clockwise orientation, but $(C_1,C_3,C_4)$ does not as shown in Figure \ref{134}.

\begin{figure}[h]
\begin{center}
\includegraphics[width=120pt]{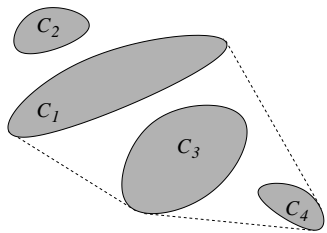}
  \caption{$(C_1,C_3,C_4)$ does not have a strong-clockwise orientation.}
  \label{134}
 \end{center}
\end{figure}

\vspace{0.1cm}
\noindent $4$. We proved that the vertex online Ramsey number $V_2(q,n)$ and the lattice game number $L(q,n)$ are equal, asymptotically determined these numbers for $q=2$, and determined the order of growth of these numbers for fixed $q$. We leave as an open problem to asymptotically determine these numbers for fixed $q>2$.

\vspace{0.1cm}
\noindent $5$. We proved that the size Ramsey number $S_2(q,n)$, which is the minimum number of edges of a graph which is $(q,n)$-path Ramsey, satisfies
$c_qn^{2q-1} \leq S_2(q,n) \leq c_q' n^{2q}$ for some constants $c_q,c_q'$ depending on $n$. It would be interesting to determine the growth of $S_2(q,n)$ for fixed $q$.

\end{document}